\documentclass{amsart}

\usepackage{amsmath,amssymb}
\usepackage{amsthm}
\usepackage{stmaryrd}
\usepackage{enumerate}
\usepackage{url}
\usepackage{wasysym}
\usepackage{color}
\usepackage[normalem]{ulem}
\allowdisplaybreaks[1]

\newtheorem{theorem}{Theorem}[section]
\newtheorem{fact}[theorem]{Fact}
\newtheorem{lemma}[theorem]{Lemma}
\newtheorem{cla}[theorem]{Claim}
\newtheorem{corollary}[theorem]{Corollary}
\newtheorem{proposition}[theorem]{Proposition}

\theoremstyle{definition}
\newtheorem{definition}[theorem]{Definition}
\newtheorem{example}[theorem]{Example}

\newtheorem{remark}[theorem]{Remark}
\newtheorem{question}[theorem]{Question}

\newcommand{\abar}{\bar{a}}

\def\tp{\operatorname{tp}}

\def\Av{\operatorname{Av}}

\def\gen{\operatorname{gen}}

\def\doag{\operatorname{doag}}
\def\cl{\operatorname{cl}}
\def\Stab{\operatorname{Stab}}

\def\msM{\mathfrak{M}_{x}(M)}
\def\mA{\mathfrak{M}_{x}(A)}

\def\cL{\mathcal{L}}
\def\cU{\mathcal{U}}

\def\Av{\operatorname{Av}}

\def\conv{\operatorname{conv}}

\def\supp{S}
\def\inva{\operatorname{inv}}

\def\Ind{\setbox0=\hbox{$x$}\kern\wd0\hbox to 0pt{\hss$\mid$\hss}
\lower.9\ht0\hbox to 0pt{\hss$\smile$\hss}\kern\wd0}

\def\Notind{\setbox0=\hbox{$x$}\kern\wd0\hbox to 0pt{\mathchardef
\nn=12854\hss$\nn$\kern1.4\wd0\hss}\hbox to
0pt{\hss$\mid$\hss}\lower.9\ht0 \hbox to 0pt{\hss$\smile$\hss}\kern\wd0}

\def\ind{\mathop{\mathpalette\Ind{}}}

\title{Definable convolution and idempotent Keisler measures}
\author[A. Chernikov]{Artem Chernikov}

\author[K. Gannon]{Kyle Gannon}

\address{Department of Mathematics\\
University of California Los Angeles\\
Los Angeles, CA 90095-1555, USA}
\email{chernikov@math.ucla.edu}

\address{Department of Mathematics\\
University of Notre Dame\\
Notre Dame, IN, 46656, USA}

\email{kgannon1@nd.edu}

\begin{document}

\begin{abstract} We initiate a systematic study of the  convolution operation on Keisler measures, generalizing the work of Newelski in the case of types. Adapting results of Glicksberg, we show that the supports of definable and finitely satisfiable (or just definable, assuming NIP) measures are nice semigroups, and classify idempotent measures in stable groups as invariant measures on type-definable subgroups.  We establish left-continuity of the convolution map in NIP theories, and use it to show that the convolution semigroup on finitely satisfiable measures is isomorphic to a particular Ellis semigroup in this context.
\end{abstract}

\maketitle
\vspace{-15pt} 

\section{Introduction} 

Various notions and ideas from topological dynamics were introduced  into the model-theoretic study of definable group actions by Newelski \cite{N1,N2}. A fundamental observation is that certain spaces of types over a definable group carry a natural algebraic structure of a (left-continuous) semigroup, with respect to the ``independent product'' of types. In a rather wide context, this operation can be extended from types to  general \emph{Keisler measures} on a definable group
(i.e.~finitely additive probability measures on the Boolean algebra of definable subsets), where it corresponds to \emph{convolution} of measures.
We first recall the classical setting. When $G$ is a locally-compact topological group, then the space of regular Borel probability measures on $G$ is equipped with the convolution product: if $\mu$ and $\nu$ are bounded measures on $G$, then their product is the measure $\mu*\nu$ on $G$ defined via
\begin{equation*}
\mu * \nu(A) = \int_{y \in G} \int_{x \in G} \chi_{A}(x\cdot y) d\mu(x) d\nu(y),
\end{equation*} for an arbitrary Borel set $A \subseteq G$ (where $\chi_A$ is the characteristic function of $A$). And a measure $\mu$ is \emph{idempotent} if $\mu * \mu = \mu$. 
A classical theorem of Wendel \cite{Wendel} shows that if $G$ is a compact topological group and $\mu$ is a regular Borel probability measure on $G$, then $\mu$ is idempotent if and only if  the support of $\mu$ is a compact subgroup of $G$,  and the restriction of $\mu$ to this subgroup is the (bi-invariant) Haar measure. Wendel's result was extended to locally compact abelian groups by Rudin \cite{Rudin} and Cohen \cite{Cohen}, and this line of research continued into the study of the structure of idempotent measures on (semi-)topological semigroups, in particular in the work of Glicksberg \cite{Glicksberg1, Glicksberg2} and Pym \cite{Pym1,Pym2}.

%

In this paper we consider the counterpart of these developments in the definable category, i.e.~ for definable groups and Keisler measures on them. In particular, we aim to address the following questions.
\begin{enumerate}
\item[(Q1)] Under what conditions the convolution product of two global Keisler measures can be defined? 
\item[(Q2)] What algebraic structures arise from idempotency of a Keisler measure? 
\item[(Q3)] Is there a connection between the convolution semigroups of Keisler measures and Ellis semigroups?
\end{enumerate}

We begin by reviewing some (mostly standard) material on Keisler measures in Section \ref{sec: prelims on Keisl meas}: we recall various classes of measures (invariant, Borel-definable, finitely satisfiable, finitely approximable, smooth), summarize the relationship between them (in general, as well as in NIP and stable theories) and discuss supports of measures. In particular, in Proposition \ref{convfs} we give a topological characterization of the space of measures finitely satisfiable over a small model $M$,  and in Lemma \ref{lem: sup on inv types} we make a couple of observations on \emph{invariantly supported} measures (i.e.~global measures such that all types in their support are (automorphism-)invariant over a fixed small model).

In Section \ref{sec: ext prod of measures} we extend the usual product  $\otimes$ of Borel-definable measures to a slightly larger context. Namely, when defining $\mu \otimes \nu$, we only require the level functions of the measure $\mu$  to be Borel \emph{restricted to the support of $\nu$} (Definition \ref{Borel}). It is equivalent to the standard definition when $\mu$ is Borel-definable, but allows one to evaluate the product of an arbitrary invariant measure $\mu$ with an arbitrary type $p$ for example (and this extends the usual independent product of invariant types, see Proposition \ref{prop: gen prod ext type prod}).
In relation to (Q1), in Section \ref{sec: def of conv} we define the convolution operation on \emph{$*$-Borel pairs} of Keisler measures in terms of this generalized product of measures (Definition \ref{def: conv})    and observe some of its basic properties, in particular that it extends the independent product of arbitrary invariant types in a group (Proposition \ref{Newelski}).

In Section \ref{sec: idemp}, we begin investigating idempotent Keisler measures. In Proposition \ref{wow} we observe that every invariant measure on a type-definable subgroup is idempotent (the extended $\otimes$-product is needed for this to hold without any definability assumptions on the invariant measure). Mirroring the classical situation in Wendel's theorem, the expectation is that in tame contexts all idempotent measures should arise in this way. In the case of a definably amenable NIP group, invariant measures were classified in \cite{CS}. We observe in Proposition \ref{prop: bdd index def am NIP} that a type-definable subgroup of bounded index of a definably amenable NIP group is still definably amenable (and the analysis from \cite{CS} extends to it). We also point out that, as a consequence of Wendel's theorem, idempotent measures finitely supported on realized types correspond to finite subgroups (Proposition \ref{prop: finite realized support}); and that in an abelian NIP group, the class of idempotent generically stable measures is closed under convolution (Proposition \ref{prop: commute preserves idemp}).

In Section \ref{sec: supports}, we study the supports of idempotent Keisler measures (question (Q2) above). In the proof of Wendel's theorem (as well as Glicksberg's proof in the abelian semitopological semigroup case \cite{Glicksberg2}), an idempotent regular Borel measure $\mu$ is associated to a closed subgroup given by its support. In particular, $\supp(\mu)$ is a closed group and $\mu|_{\supp(\mu)}$ is its associated (bi-invariant) Haar measure. In the general model-theoretic context the situation is not as nice (see Examples \ref{exa: 1} and \ref{circle}). However, adapting some of Glicksberg's work to our context, we show that if $\mu$ is definable, invariantly supported and idempotent, then $\left(\supp(\mu),* \right)$ (with respect to the usual independent product of invariant types) is a compact, left-continuous  semigroup with no closed two-sided ideals (Corollary \ref{qcont} and Theorem \ref{two}). This assumption is satisfied when $\mu$ is a dfs measure in an arbitrary theory, or when $\mu$ is an arbitrary definable measure in an NIP theory.
We also deduce that if $\supp(\mu)$ has no proper closed left ideals, then $\mu$ is ``generically'' invariant restricting to its support (Corollary \ref{cor: meas on sup is invariant}). It follows that  in abelian stable groups, the supports of the idempotent measures are precisely the closed subgroups of the convolution semigroup on the space of types (Corollary \ref{cor: supp stab ab}); which leads to a quick description of idempotent measures in strongly minimal groups (Example \ref{exa: str min}).

In Section \ref{sec: stable groups} we classify idempotent measures on a stable group, demonstrating that they are precisely the invariant measures on its type-definable subgroups. More precisely, every idempotent measure is the unique invariant Keisler measure on its own (type-definable) stabilizer. Our proof relies on the results of the previous section and a variant of Hrushovski's group chunk theorem due to Newelski \cite{N3}.

Concerning question (Q3), it was observed by Newelski \cite{N1} that the convolution semigroup $(S_{x}(\mathcal{G},G),*)$ on the space of global types finitely satisfiable in a small model $G \prec \mathcal{G}$ is isomorphic to the enveloping \emph{Ellis semigroup} $E(S_{x}(\mathcal{G},G),G)$ of the action of $G$ on this space of types.
Ellis semigroups for definable group actions in the context of NIP theories were previously considered in e.g.~\cite{MR3245144, CS}, to which we refer for a general discussion.
In Section \ref{sec: Ellis calc}, under the NIP assumption, we obtain an analogous description for the convolution semigroup $(\mathfrak{M}_{x}(\mathcal{G},G),*)$ on the space of global Keisler measures finitely satisfiable in a small model. Namely, in Theorem \ref{thm: Ellis grp iso} we show that it is isomorphic to the Ellis semigroup  $E(\mathfrak{M}_{x}(\mathcal{G},G),\conv(G))$ of the action of $\conv(G)$, the convex hull of $G$ in the space of global measures finitely satisfiable on $G$,  on this space of measures (see Remark \ref{rem : without conv} on why the convex hull is necessary). Our proof relies in particular on left-continuity of convolution of invariant measures in NIP theories established in Section \ref{sec: left-cont of conv} using approximation arguments with smooth measures.

\subsection*{Acknowledgements}
We thank the referee for many helpful suggestions on improving the paper. This work constitutes part of the Ph.D dissertation of the second named author.
We thank Sergei Starchenko for several helpful conversations on the topics considered here. 
Both authors were partially supported by the NSF CAREER grant DMS-1651321, and Gannon was additionally supported by the NSF conference grant DMS-1922826.

\section{Preliminaries on Keisler measures}\label{sec: prelims on Keisl meas}

\subsection{Basic facts about Keisler measures} For the majority of this article, we focus on global Keisler measures and their relationship to small elementary submodels. In this section we recall some of the material from \cite{Keisl, NIP1, NIP2, NIP3, NIP5, Gannon}, and refer to e.g.~\cite[Chapter 7]{Guide} for a more detailed introduction to the subject, or \cite{StBourb, CheSurv} for a survey.

Given $r_1, r_2 \in \mathbb{R}$ and $\varepsilon \in \mathbb{R}_{>0}$, we  write $r_1 \approx_{\varepsilon} r_2 $ if $|r_1 - r_2| < \varepsilon$. Let $T$ be a first order theory in a language $\mathcal{L}$ and assume that $\mathcal{U}$ is a sufficiently saturated model of $T$ (we make no assumption on $T$ unless explicitly stated otherwise). In this section, we write $x,y,z, \ldots$ to denote arbitrary finite tuples of variables. If $x$ is a tuple of variables and $A \subseteq \mathcal{U}$, then $\mathcal{L}_{x}(A)$ is the collection of formulas with free variables in $x$ and parameters from $A$, up to logical equivalence (which we identify with the corresponding Boolean algebra of definable subsets of $\cU^x$). We write $\mathcal{L}_{x}$ for $\mathcal{L}_{x}(\emptyset)$. Given a partitioned formula $\varphi(x;y)$, we let $\varphi^*(y;x) := \varphi(x;y)$ be the partitioned formula with the roles of $x$ and $y$ reversed.
As usual, $S_x(A)$ denotes the space of types over $A$, and if $A \subseteq B \subseteq \cU$ then $S_x(B,A)$ (respectively, $S^{\inva}_x(B,A)$) denotes the closed set of types in $S_x(B)$ that are finitely satisfiable in $A$ (respectively, invariant over $A$). 
For any set $A \subseteq \mathcal{U}$, a \emph{Keisler measure} over $A$ in variables $x$ is a finitely additive probability measure on $\mathcal{L}_{x}(A)$. We denote the space of Keisler measures over $A$ (in variables $x$) as $\mA$. Every element of $\mA$ is in unique correspondence with a regular Borel probability measure on the space $S_{x}(A)$, and we will routinely use this correspondence. If $M_0 \preceq M \preceq \cU$ are small models, then there is an obvious restriction map $r_{0}$ from $\msM$ to $\mathfrak{M}_{x}(M_0)$ and we denote $r_{0}(\mu)$ simply as $\mu|_{M_0}$. Conversely, every $\mu \in \mathfrak{M}_{x}(M_0)$ admits an extension to some $\mu' \in \mathfrak{M}_{x}(M)$ (not necessarily a unique one).

The space $\mA$ is a compact Hausdorff space with the topology induced from $[0,1]^{\mathcal{L}_{x}(A)}$. This is the coarsest topology on the set $\mA$ such that for any continuous function $f:S_{x}(A) \to \mathbb{R}$, the map $\mu \to \int f d\mu$ is continuous. If $M_0 \preceq M$, then under this topology, the restriction map $r_0$ is continuous.  We identify every $p \in S_x(A)$ with the corresponding Dirac measure $\delta_p \in \mA$, and under this identification $S_x(A)$ is a closed subset of $\mA$.

We recall several important properties of global measures that will make an appearance in this article.
\begin{definition}\label{def: props of measures} Let $\mu \in \mathfrak{M}_{x}(\mathcal{U})$ be a global Keisler measure.
\begin{enumerate}
\item  $\mu$ is \emph{invariant} if there is a small model $M \prec \mathcal{U}$ such that for any partitioned $\cL(M)$-formula $\varphi(x;y)$ and any $b,b'\in \mathcal{U}^y$, if $b\equiv_M b'$ then $\mu(\varphi(x;b))=\mu(\varphi(x;b'))$. In this case, we say $\mu$ is \emph{$M$-invariant}.
 We let $\mathfrak{M}^{\inva}_x(\cU, M)$ denote the closed set of all $M$-invariant measures in $\mathfrak{M}_x(\cU)$.
\item Assume that $\mu$ is $M$-invariant and $\varphi(x;y)$ is a partitioned $\cL(M)$-formula. We define the map $F_{\mu,M}^{\varphi}:S_{y}(M) \to [0,1]$ by $F_{\mu,M}^{\varphi}(q)=\mu(\varphi(x;b))$, where $b\models q$ (this is well-defined by $M$-invariance of $\mu$). 

\noindent We will often write $F_{\mu}^{\varphi}$ instead of $F_{\mu,M}^{\varphi}$ when the base model $M$ is clear from the context.
\item $\mu$ is \emph{Borel-definable} (respectively, \emph{definable}) if there is $M\prec\cU$ such that $\mu$ is $M$-invariant and for any partitioned $\cL(M)$-formula $\varphi(x;y)$, the map $F_{\mu, M}^{\varphi}$ is Borel-measurable (respectively, continuous). In this case, we say that $\mu$ is \emph{Borel-definable over $M$} (respectively, \emph{definable over $M$}). 
\item $\mu$ is \emph{finitely satisfiable in $M\prec \mathcal{U}$}  if for any $\cL_x(\mathcal{U})$-formula $\varphi(x)$, if $\mu(\varphi(x))>0$ then $\cU\models\varphi(a)$ for some $a\in M^x$.  We let $\mathfrak{M}_{x}(\mathcal{U},M)$ denote the closed set of measures in $\mathfrak{M}_{x}(\mathcal{U})$ which are finitely satisfiable in $M$.
\item $\mu$ is \emph{dfs} if there is $M\prec\cU$ such that $\mu$ is both definable over $M$ and finitely satisfiable in $M$. Similarly, if this is the case, we say that $\mu$ is \emph{dfs over $M$}.
\item Given $\overline{a} \in (\cU^{x})^{<\omega}$, with $\overline{a} = (a_1,...,a_n)$, the associated \emph{average measure} $\Av_{\overline{a}} \in \mathfrak{M}_{x}(\mathcal{U})$ is defined by
\begin{equation*} \Av_{\overline{a}}(\varphi(x)) := \frac{|\{i: \mathcal{U} \models \varphi(a_i)\}| }{n}
\end{equation*}
for any $\varphi(x) \in \mathcal{L}_{x}(\cU)$.

\item $\mu$ is \emph{finitely approximated} if there is $M\prec\cU$ such that for any $\cL(M)$-formula $\varphi(x;y)$ and any $\varepsilon \in \mathbb{R}_{>0}$, there exist  $n \in \mathbb{N}_{\geq 1}$ and $\abar\in (M^x)^n$ such that for any $b\in\cU^{y}$, $\mu(\varphi(x;b)) \approx_{\varepsilon} \Av_{\abar}(\varphi(x;b))$. In this case, we call $\abar$ a \emph{$(\varphi,\varepsilon)$-approximation for $\mu$}, and we  say $\mu$ is \emph{finitely approximated in $M$}. 

\item $\mu$ is \emph{smooth} if there exists a small model $M \prec \mathcal{U}$ such that for any $N$ with $M \preceq N \preceq \cU$, there exists a unique measure $\mu' \in \mathfrak{M}_{x}(N)$ such that $\mu'|_M = \mu|_{M}$. In this case, we say that $\mu$ is \emph{smooth over $M$}.

\end{enumerate}  
\end{definition}

These properties are related as follows.
\begin{fact}\label{fac: props of measures relation} 
	\begin{enumerate}
\item In any theory $T$, given $\mu \in \mathfrak{M}_{x}(\mathcal{U})$, over any given $M \prec \cU$: 
\begin{enumerate}
	\item $\mu$ is smooth $\Rightarrow$ $\mu$ is finitely approximated \cite[Corollary 2.6]{NIP3};
	\item  $\mu$ is finitely approximated $\Rightarrow$ $\mu$ is dfs (e.g.~see \cite[Proposition 4.12]{Gannon});
	\item $\mu$ is definable $\Rightarrow$  $\mu$ is Borel-definable;
	\item if $\mu$ is either Borel-definable or finitely satisfiable, then $\mu$ is invariant.
\end{enumerate}  
\item Assuming $T$ is NIP, given $\mu \in \mathfrak{M}_{x}(\mathcal{U})$, over any $M \prec \cU$ we have additionally:
\begin{enumerate}
\item $\mu$ is invariant $\Rightarrow$ $\mu$ is Borel-definable (\cite[Corollary 4.9]{NIP2}, or \cite[Proposition 7.19]{Guide});
\item $\mu$ is dfs $\Rightarrow$ $\mu$ is finitely approximated \cite[Theorem 3.2]{NIP3}.
\end{enumerate}
\item Assuming $T$ is stable, given any $\mu \in \mathfrak{M}_{x}(\mathcal{U})$ we have moreover:
\begin{enumerate}
\item $\mu$ is finitely approximated (see e.g.~\cite[Lemma 4.3]{NIP5} for a direct proof);
\item for every $\mathcal{L}$-formula $\varphi(x;y)$, there exist types $(p_i)_{i \in \omega}$ in $S_x(\cU)$ and $(r_i)_{i\in \omega}, r_i \in [0,1]$ such that $\sum r_i = 1$, and taking $\mu' := \sum r_i \cdot p_i$ we have $\mu(\varphi(x;b)) = \mu'(\varphi(x;b))$ for all $b \in \cU^y$ \cite[Lemma 1.7]{Keisl};
\item If $T$ is $\omega$-stable, then there exist $(p_i)_{i \in \omega}$ in $S_x(\cU)$ and $(r_i)_{i\in \omega}, r_i \in [0,1]$ such that $\sum r_i = 1$ and $\mu = \sum r_i \cdot p_i$ (same as the proof of \cite[Lemma 1.7]{Keisl}, using boundedness of the global rank).
\end{enumerate}
\end{enumerate}
\end{fact}

We have the following characterization of definability (see e.g.~\cite[Proposition 4.4]{Gannon}).
\begin{fact}\label{fac: chars of def meas} The following are equivalent for $\mu \in \mathfrak{M}_{x}(\mathcal{U})$ and $M \preceq \cU$.
\begin{enumerate}
\item The measure $\mu$ is definable over $M$.
\item For any partitioned $\cL(M)$-formula $\varphi(x;y)$ and any $\varepsilon > 0$, there exist formulas $\Phi_{1}(y),...,\Phi_{n}(y)$ such that each $\Phi_{i}(y) \in \mathcal{L}_{y}(M)$, the collection $\{\Phi_{i}(\cU): i \leq n\}$ forms a partition of $\mathcal{U}^{y}$, and if $\models \Phi_{i}(c) \wedge \Phi_{i}(c')$, then $|\mu(\varphi(x,c)) - \mu(\varphi(x,c'))| < \varepsilon$.
\item For every partitioned formula $\varphi(x;y) \in \mathcal{L}(M)$ and every $n \in \mathbb{N}_{\geq 1}$ there exist some $\mathcal{L}_y(M)$-formulas $\Phi^{\varphi,\frac{1}{n}}_i(y)$ with $i \in I_n := \{0, \frac{1}{n},   \frac{2}{n}, \ldots, \frac{n-1}{n}, 1\}$ such that:
		\begin{enumerate}
			\item the collection $\{ \Phi^{\varphi, \frac{1}{n}}_i(\cU)  : i \in I_n\}$ forms a covering of $\cU_y$ (but not necessarily a partition);
			\item For every $i \in I_n$ and $b \in \cU_y$, if $\cU \models \Phi^{\varphi,\frac{1}{n}}_i(b)$ then $|\mu(\varphi(x,b)) - i| < \frac{1}{n}$.
		\end{enumerate}
\end{enumerate}
\end{fact}

This easily implies the following.

\begin{fact}\label{fac: unique ext of def meas}
If $M \preceq N \prec \cU$ and $\mu \in \mathfrak{M}_x(N)$ is definable over $M$, then there exists a unique extension $\mu' \in \mathfrak{M}_x(\cU)$ of $\mu$ which is definable over $N$, denoted $\mu|_{\cU}$ (it is then automatically definable over $M$ and given by the same definition schema $\Phi$ for $\mu$ as in Fact \ref{fac: chars of def meas}).
\end{fact}

In an NIP theory, every measure over a small model can be extended to a smooth measure over a slightly larger elementary extension (\cite[Theorem 3.16]{Keisl}, or \cite[Proposition 7.9]{Guide}).
\begin{fact}\label{fac: NIP ext to smooth}
Let $T$ be an NIP theory. Let $M \prec \mathcal{U}$ and $\mu \in \mathfrak{M}_{x}(M)$. Then $\mu$ admits a smooth extension. I.e., there exist some $\nu \in \mathfrak{M}_{x}(\mathcal{U})$ and some small $M \preceq N \prec \cU$ such that $\nu$ is smooth over $N$ and $\nu|_{M} = \mu$. 
\end{fact}

\begin{definition}
	Given a Keisler measure $\mu \in \mathfrak{M}_x(A)$, the \emph{support of $\mu$} is 
\begin{equation*}
    \supp(\mu) = \{ p \in S_{x}(A): \mu(\varphi(x)) > 0 \text{ for any } \varphi(x) \in p\}. 
\end{equation*}
Types in $\supp(\mu)$ are sometimes called \emph{weakly random} with respect to $\mu$ in the literature.
\end{definition}
We recall some properties of supports, with  proofs for the sake of completeness.
\begin{proposition}\label{prop: supp 1} Let $\mu \in \mA$. 
\begin{enumerate}
	\item Then for any  $\varphi(x) \in \mathcal{L}_{x}(A)$ such that $\mu(\varphi(x)) > 0$, there exists some $q \in \supp(\mu)$ such that $\varphi(x) \in q$. In particular, $\supp(\mu) \neq \emptyset$. 
	\item $\supp(\mu)$ is a closed subset of $S_{x}(A)$ and $\mu(\supp(\mu)) = 1$ (and $S(\mu)$ is the smallest set of types under inclusion with this property).
\end{enumerate}

\end{proposition} 
\begin{proof}

(1) Without loss of generality, $\varphi(x) \equiv x=x$. Otherwise, we reiterate the proof with the normalization of $\mu$ to the definable set $\varphi(x)$, i.e.~considering the Keisler measure $\mu'$ defined by $\mu'(\psi(x)) := \frac{\mu(\psi(x) \land \varphi(x))}{\mu(\varphi(x))}$ for all $\psi(x)$. Assume that $\supp(\mu) = \emptyset$, then for every type $p \in S_{x}(A)$, there exists some $\varphi_{p}(x) \in p$ such that $\mu(\varphi_{p}(x)) = 0$. Then, $\mu(\neg \varphi_{p}(x)) =1$ for every $p \in S_{x}(A)$, hence  for any $n$ and $p_1,...,p_n \in S_{x}(A)$, we have $\bigcap_{i=1}^{n} \neg \varphi_{p_i}(x) \neq \emptyset$.  Then $K = \bigcap_{p \in S_{x}(A)} \neg \varphi_{p}(x) \neq \emptyset$ by compactness of $S_x(A)$. But if $q \in K$, then in particular $\neg \varphi_{q}(x) \in q$ --- a contradiction.

(2) Assume that $p \not \in \supp(\mu)$. Then, there exists a formula $\varphi_{p}(x)$ such that $\varphi_{p}(x) \in p$ and $\mu(\varphi_{p}(x)) =0$. Then,
\begin{equation*} 
S_x(A) \backslash \supp(\mu) = \bigcup_{p \not \in  \supp(\mu)} \varphi_{p}(x).
\end{equation*}
Therefore, $\supp(\mu)$ is closed. Assume that $\mu(S_{x}(M) \backslash \supp(\mu)) > 0$. By regularity of $\mu$, there exists a clopen $C \subseteq S_{x}(A) \backslash \supp(\mu)$ with positive measure. But by (1) we must have $C \cap \supp(\mu) \neq \emptyset$, a contradiction. 
\end{proof}

\begin{proposition}\label{coheir} Let $A \subseteq B \subseteq \cU$ and $\mu \in \mathfrak{M}_x(B)$ be arbitrary. Let $r: S_{x}(B) \to S_{x}(A),  q \mapsto q|_{A}$ be the restriction map. Then:
\begin{enumerate}
\item $r(\supp(\mu)) = \supp(\mu|_{A})$;
	\item the measure $\mu|_{A}$ is the pushforward of $\mu$ along $r$, i.e.~$r^{*}(\mu) = \mu|_{A}$.
\end{enumerate}
\end{proposition}

\begin{proof} (1) The map $r$ is a continuous surjection between compact Hausdorff spaces. By Proposition \ref{prop: supp 1}(2), $r(\supp(\mu))$ is compact (hence, closed), as the continuous image of a compact set. Clearly  $r(\supp(\mu)) \subseteq \supp(\mu|_A)$, and as $r(\supp(\mu))$ is closed it suffices to show that $r(\supp(\mu))$ is a dense subset of $\supp(\mu|_A)$. Indeed, assume that $\varphi(x) \in \mathcal{L}_{x}(A)$ and $\varphi(x) \cap \supp(\mu|_{A}) \neq \emptyset$. Then $\mu|_{A}(\varphi(x)) > 0$, hence $\mu(\varphi(x)) > 0$, and by Proposition \ref{prop: supp 1}(1) there exists some $q \in \supp(\mu)$ with  $\varphi(x) \in q$. Hence $\varphi(x) \in r(q)$, and so $r(\supp(\mu)) \cap \varphi(x) \neq \emptyset$. And (2) is clear.
\end{proof}

\begin{definition}\label{def: inv sup meas}
	We say that $\mu \in \mathfrak{M}_x(\cU)$ is \emph{invariantly supported} if there exists some small model $M \prec \cU$ such that every type $p \in \supp(\mu)$ is $M$-invariant.
\end{definition}
\begin{lemma} \label{lem: sup on inv types} Let $\mu \in \mathfrak{M}_x(\cU)$.
	\begin{enumerate}
	\item If $\mu$ is finitely satisfiable, then $\mu$ is invariantly supported.
		\item If $\mu $ is invariantly supported, then $\mu$ is invariant.
		\item If $T$ is NIP, then $\mu$ is invariant if and only if it is invariantly supported.
		\item In some theory, there exist a definable measure $\mu \in \mathfrak{M}_x(\cU)$ and $p \in \supp(\mu)$ such that $p$ is not invariant (over any small set).
		\end{enumerate}
	\end{lemma}
\begin{proof}
(1) Clearly if $\mu$ is finitely satisfiable over a small model $M \prec \cU$, then every $p \in \supp(\mu)$ is also finitely satisfiable in $M$.

	(2) Let $M \prec \cU$ be a small model such that every $p \in \supp(\mu)$ is invariant over $M$. If $\mu$ is not invariant over $M$, then there exist some $\varphi(x,y) \in \mathcal{L}_{xy}$ and some $b \equiv_M b'$ in $\cU^y$ such that $\mu(\varphi(x,b)) \neq \mu(\varphi(x,b'))$. But then $\mu(\varphi(x,b) \triangle \varphi(x,b')) > 0$, hence $\varphi(x,b) \triangle \varphi(x,b') \in p$ for some $p \in \supp(\mu)$ by Proposition \ref{prop: supp 1} --- contradicting $M$-invariance of $p$.
	
	(3) $(\Leftarrow)$ holds by (2). For $(\Rightarrow)$, we note that if $\mu$ is invariant over $M \prec \cU$, then every global type $p \in \supp(\mu)$ does not divide over $M$ (given $\varphi(x,b) \in p$ and an $M$-indiscernible sequence $(b_i)_{i \in \omega}$ in $\cU^y$ such that $b_i \equiv_M b$, we have that $\mu(\varphi(x,b_i)) = \mu(\varphi(x,b)) =: \varepsilon > 0$ for all $i$; but then $\mu(\bigwedge_{i < k} \varphi(x,b_i)) > 0$ for every $k\in \omega$, so in particular $\neq \emptyset$, by a standard probability lemma, see e.g. \cite[Lemma 7.5]{Guide}), hence $p$ is invariant over $M$ by \cite[Proposition 2.1(ii)]{NIP2}.
	
	(4) Let $T$ be the theory of the random graph, in a language with a single binary relation. We let $\mu 
	\left(\bigwedge_{i<k} E(x,b_i)^{t_i} \right) = \frac{1}{2^k}$ for every $k \in \omega$, pairwise distinct $b_i \in \cU$ and $t_i \in \{0,1 \}$, and $\mu(x=b) = 0$ for every $b \in \cU$. By quantifier elimination,  this determines a measure  $\mu \in \mathfrak{M}_x(\cU)$. This $\mu$ is clearly definable over $\emptyset$, and the support of $\mu$ consists of all non-realized types in $S_x(\cU)$. However, it is easy to construct by transfinite induction a non-realized type $p \in S_x(\cU)$ which is not invariant over any small $M \prec \cU$.
\end{proof}
The space of measures $\mathfrak{M}_x(\cU)$ can be naturally viewed as a closed convex subset of a real topological vector space (of all bounded real-valued measures). Given $M \prec \cU$, we identify $M^x$ with the set $\{\delta_a : a \in M^x \} \subseteq \mathfrak{M}_x(\cU)$, and let $\conv(M^x)$ denote the convex hull. We have the following topological characterization of finite satisfiability for measures. 

\begin{proposition}\label{convfs} Let $\mu \in \mathfrak{M}_{x}(\mathcal{U})$ and let $M \prec \cU$ be a small model. Then $\mu$ is finitely satisfiable in $M$ if and only if $\mu$ is in the closure of $\conv(M^{x})$ (viewed as a subset of $\mathfrak{M}_{x}(\mathcal{U}))$.
\end{proposition} 

\begin{proof} Assume $\mu$ is finitely satisfiable in $M$. Let $U$ be a basic open subset of $\mathfrak{M}_{x}(\mathcal{U})$ containing $\mu$. Say
\begin{equation*}
    U = \bigcap_{i=1}^n \left\{ \mu' \in \mathfrak{M}_{x}(\cU): r_i < \mu'(\varphi_{i}(x)) < s_i\right\}
\end{equation*}
for some $n \in \mathbb{N}$, $\varphi_1(x),...,\varphi_n(x) \in  \mathcal{L}_x(\mathcal{U})$ and $r_{i},...,r_{n}, s_1,...,s_n \in [0,1]$.
The collection $\{\varphi_1(x),...,\varphi_n(x)\}$ generates a finite Boolean subalgebra of $\mathcal{L}_x(\mathcal{U})$. Let $\theta_1(x),...,\theta_m(x)$ be its atoms, and let $\Theta := \{\theta_j(x): \mu(\theta_j(x)) > 0\}$. As $\mu$ is finitely satisfiable in $M$, for each $\theta_j(x) \in \Theta$, there exists some $a_j \in M^{x}$ such that $\models \theta_j(a_j)$. Let $\nu := \sum_{\theta_j \in \Theta} \mu(\theta_j(x))\delta_{a_j} \in \mathfrak{M}_{x}(\mathcal{U})$.
Then we have $\mu(\varphi_i(x)) = \nu(\varphi_i(x))$ for all $1 \leq i \leq n$ (note that $a_j \models \theta_i \iff  i=j$), so $\nu \in U \cap \conv(M^{x})$. Hence $\mu \in \cl(\conv(M^{x}))$. 

Conversely, suppose $\mu \in \cl(\conv(M^{x}))$ and let $\psi(x) \in \mathcal{L}_{x}(\mathcal{U})$ be such that $\mu(\psi(x)) > 0$. Consider the open set $U_{\psi} := \{\nu \in \mathfrak{M}_{x}(\mathcal{U}): 0 < \nu(\psi(x))\}$ containing $\mu$. Since $\mu$ is in the closure of $\conv(M^{x})$, there exists some $\mu_{\psi}=\sum_{i=1}^{n} r_i \delta_{a_i}$, where $a_i \in  M^{x}$ for all $i$ and $\mu_{\psi} \in U_{\psi}$. But then $ \mathcal{U} \models \psi(a_i)$ for at least one $i$. 
\end{proof}

\section{Definable convolution and idempotent measures}

\subsection{Extended product of measures}\label{sec: ext prod of measures} We begin by defining a slight generalization of the product of measures that encompasses both the usual independent product of Borel-definable measures and the standard Morley product of invariant types (without any definability assumptions), and also allows to take products of $G$-invariant measures in arbitrary theories. This is accomplished by slightly tweaking the domain of the integral in the usual definition of the $\otimes$-product. 

\begin{definition}\label{Borel} Let $\mu \in \mathfrak{M}_{x}(\mathcal{U})$, $\nu \in \mathfrak{M}_{y}(\mathcal{U})$, and $\varphi(x,y,\overline{c}) \in \mathcal{L}_{xy}(\mathcal{U})$. We say that the triple $(\mu,\nu,\varphi)$ is \emph{Borel} if there exists some $N \prec \mathcal{U}$ such that:
\begin{enumerate}
    \item $\overline{c} \subseteq N$;
    \item for any $q \in \supp(\nu|_{N})$ and $d,d' \in \mathcal{U}^{y}$ with $d,d'\models q$, we have that $\mu(\varphi(x,d,\overline{c})) = \mu(\varphi(x,d',\overline{c}))$;
    \item the map $F_{\mu,N}^{\varphi}: \supp(\nu|_{N}) \to [0,1]$ is Borel, where $F_{\mu,N}^{\varphi}(q) = \mu(\varphi(x,d,\overline{c}))$ for some/any $d \models q$. 
\end{enumerate}
We say that the ordered pair $(\mu,\nu)$ is \emph{Borel} if $(\mu,\nu,\varphi)$ is Borel for any $\varphi(x,y,\overline{c}) \in \mathcal{L}_{xy}(\mathcal{U})$. 
\end{definition} 

\begin{definition}\label{def: gen prod} Assume that $(\mu,\nu)$ is Borel. Then we define the product measure $\mu \tilde{\otimes} \nu \in \mathfrak{M}_{xy}(\cU)$ as follows: given an arbitrary formula $\varphi(x,y,\bar{c}) \in \mathcal{L}_{xy}(\mathcal{U})$,
 let $N$ be any small elementary submodel of $\mathcal{U}$ witnessing that $(\mu,\nu,\varphi)$ is Borel (as in Definition \ref{Borel}); we define
\begin{equation*}
    \mu \tilde{\otimes} \nu(\varphi(x,y,\overline{c})) := \int_{\supp(\nu|_{N})} F_{\mu,N}^{\varphi} d\nu_{N},
\end{equation*}
with the notation from Definition \ref{Borel}, where $\nu_{N}$ is the restriction of the regular Borel measure $\nu|_N$ to the compact set $\supp(\nu|_{N})$. 
\end{definition}

We check that $\tilde \otimes$ is well-defined. 

\begin{proposition}\label{prop: int over diff mod} Assume that $(\mu,\nu,\varphi)$ is Borel. Then the value of $\mu \tilde{\otimes} \nu(\varphi(x;y,\overline{c}))$ in Definition \ref{def: gen prod} does not depend on the choice of $N$ (as in Definition \ref{Borel}). 
\end{proposition}
\begin{proof}
This proof is essentially the same as for $\otimes$ (see e.g.~\cite[Proposition 7.19]{Guide}). Assume that $(\mu,\nu,\varphi)$ is Borel with respect to both $M$ and $N$. We may assume that $M \subseteq N$ (taking a common extension). By Proposition \ref{coheir}, let $r: \supp(\nu_{N}) \to \supp(\nu_{M})$ be the restriction map; then $F_{\mu,M}^{\varphi} \circ r = F_{\mu,N}^{\varphi} $ and the pushforward of the measure $\nu_{N}$, namely $r^{*}(\nu_{N})$, is equal to $\nu_{M}$. Hence we have:
\begin{equation*}
\int_{\supp(\nu|_{M})} F_{\mu,M}^{\varphi} d(\nu_{M}) = \int_{\supp(\nu|_{M})} F_{\mu,M}^{\varphi} d r^*(\nu_{N}) = \int_{\supp(\nu|_{N})} \left( F_{\mu,M}^{\varphi} \circ r \right) d \nu_{N}
\end{equation*} 
\begin{equation*}
    = \int_{\supp(\nu|_{N})} F_{\mu,N}^{\varphi} d{\nu_{N}}.
\end{equation*}
\end{proof}

We recall the independent product of invariant types (see e.g.~\cite[Section 2.2]{Guide}).
\begin{fact}\label{fac: product of inv types}
	\begin{enumerate}
		\item Assume $M \prec \cU$ is a small submodel, $p \in S_x^{\inva}(\cU,M)$ and $\cU' \succ \cU$. There exists a unique type $p' \in S_x^{\inva}(\cU',M)$ extending $p$. Then for any $A \subseteq \cU'$, we write $p|_A$ to denote $p'|_A$.
		\item Assume that $p \in S_{x}(\mathcal{U}), q \in S_{y}(\mathcal{U})$ and $p$ is invariant. Then $p \otimes q := tp(a,b/\mathcal{U}) \in S_{xy}(\cU)$ for some/any $b \models q$ and $a \models p|_{\mathcal{U}b}$ (in some $\cU' \succ \cU$; this is well-defined by (1)). 
		\item If $p,q \in S_x(\cU, M)$ (respectively, $p,q \in S^{\inva}_x(\cU, M)$), then $p \otimes q \in S_{xy}(\cU, M)$ (respectively, $p \otimes q \in S^{\inva}_{xy}(\cU, M)$).
	\end{enumerate}
\end{fact}

The product $\tilde{\otimes}$ extends both the independent product on invariant types and the product of Borel definable probability measures in arbitrary theories.
\begin{proposition}\label{prop: gen prod ext type prod}
	\begin{enumerate}
\item Let $\mu \in \mathfrak{M}_{x}(\mathcal{U})$ and $\nu \in \mathfrak{M}_{y}(\mathcal{U})$. Assume that $\mu$ is Borel-definable. Then, $\mu \otimes \nu = \mu \tilde{\otimes} \nu$.  
\item If $\mu \in \mathfrak{M}_x(\cU)$ is invariant and $q\in S_x(\cU)$ is arbitrary, then $(\mu,\delta_q)$ is Borel and $\mu \tilde{\otimes} \delta_q$ is well-defined.
\item Let $p \in S_{x}(\mathcal{U})$ and $q \in S_{y}(\mathcal{U})$, and $p$ is invariant. Then $\delta_{p \otimes q} = \delta_p \tilde{\otimes} \delta_q$, where $p \otimes q$ is the free product (see Fact \ref{fac: product of inv types}).
\end{enumerate} 
\end{proposition}
\begin{proof}
	(1) Assume that $\mu$ is Borel-definable over $M \prec \cU$. Let $\varphi(x,y,\bar{c})$ be an arbitrary formula in $\cL_{xy}(\cU)$, and let $M' \prec \cU$ witness $(\mu, \nu, \varphi)$ is Borel. Taking $N \prec \cU$ to be an elementary extension of both $M$ and $M'$, we have that both $\mu$ is Borel-definable over $N$ and $N$ witnesses that $(\mu, \nu, \varphi)$ is Borel.
	It is then easy to see that $\int_{S_{y}(N)}F^{\varphi}_\mu d (\nu|_N)  =  \int_{\supp(\nu|_N)}F^{\varphi}_\mu d \nu_N$ as long as the integral on the left hand side is well-defined --- which is the case by Borel-definability of $\mu$.
	
	(2) Let $\varphi(x,y) \in \mathcal{L}(\cU)$, and let $N \prec \cU$ containing all the parameters from $\varphi$ be such that $\mu$ is invariant over $N$.
Note that the map $F^\varphi_{\mu}: S_y(N) \to [0,1]$ need not be Borel definable. However $\supp(\delta_q|_N)$ is a single point as $q$ is a type, hence $F^\varphi_{\mu} \restriction_{\supp(\delta_q|_N)}$ is trivially Borel.
	
	(3) By (2), $(\delta_p, \delta_q)$ is Borel. Let $N \prec \cU$ witness this, and let $b \in \cU^y, b \models q|_N$. Then 
\begin{equation*} \delta_{p} \tilde{\otimes}\delta_{q}(\varphi(x;y))=\int_{\supp(\delta_{q}|_{N})} F_{\delta_{p}}^{\varphi} d(\delta_q)_{N} = F_{\delta_{p}}^{\varphi}(q|_{N}) = \begin{cases}
\begin{array}{cc}
1 & \varphi(x, b) \in p,\\
0 & \neg \varphi(x,b) \in p.
\end{array}\end{cases} 
\end{equation*} 

That is, $\delta_{p} \tilde{\otimes}\delta_{q}(\varphi(x;y)) = 1$ if and only if $\varphi(x,y) \in \tp(a,b/N)$  for some/any $b\models q|_N$ and $a \models q|_{Nb}$. 
\end{proof}

\begin{remark}\label{rem: tilde otimes on types}
	If $p \in S_x(\cU), q \in S_y(\cU)$, we say that $(p,q)$ is Borel if $(\delta_{p}, \delta_{q})$ is Borel. In this case the proof of Proposition \ref{prop: gen prod ext type prod}(3) shows that there exists some $r \in S_{xy}(\cU)$ such that $\delta_{r} = \delta_p \tilde{\otimes} \delta_q$, we will denote it by $r := p \tilde{\otimes} q$ --- by Proposition \ref{prop: gen prod ext type prod} this extends the $\otimes$ operation on invariant types to a larger class of types. 
\end{remark}

From now on we will simply write $\otimes$ instead of $\tilde{\otimes}$ to denote this extended operation (on types and measures) when there is no ambiguity involved. 

\begin{definition}
	We say that $\mu \in \mathfrak{M}_x(\cU)$ and $\nu \in  \mathfrak{M}_y(\cU)$ ($\otimes$-)\emph{commute} if both $(\mu,\nu)$ and $(\nu,\mu)$ are Borel, and $\mu \otimes \nu = \nu \otimes \mu$.
\end{definition}

We recall  some facts about the $\otimes$ operation and commuting measures.
 \begin{fact}\label{fac: meas commute} Assume that $\mu \in \mathfrak{M}_{x}(\mathcal{U}), \nu \in \mathfrak{M}_{y}(\mathcal{U}), \lambda \in \mathfrak{M}_{z}(\mathcal{U})$ and $M \prec \cU$.
 \begin{enumerate}
 	\item \cite[Theorem 2.5]{NIP3}\label{fac: smooth commute} Assume that $\mu$ is Borel-definable over $M$ and $\nu$ is smooth over $M$. Then, for any $\varphi(x;y) \in \mathcal{L}_{xy}(M)$, we have that,
\begin{equation*} \int_{S_{x}(M)} F_{\mu}^{\varphi} d(\nu|_{M}) = \int_{S_{y}(M)} F_{\nu}^{\varphi^*}d(\mu|_M).
\end{equation*} 
In particular, $\mu \otimes \nu = \nu \otimes \mu$. 

\item (\cite[Proposition 2.13]{NIP5} or \cite[Proposition 2.10]{CG}) If $\mu$ and $\nu$ are finitely approximated over $M$, then $\mu \otimes \nu = \nu \otimes \mu$.
\item \cite[Proposition 7.30]{Guide} If $T$ is NIP, $\mu$ is dfs over $M$ and $\nu$ is invariant over $M$, then  $\mu \otimes \nu = \nu \otimes \mu$.
\item \cite[Corollary 1.3]{GanCon2} If $\mu$ and $\nu$ are smooth over $M$, then $\mu \otimes \nu$ is also smooth over $M$.
\item If $T$ is NIP and $\mu,\nu$ are invariant, then $\mu \otimes (\nu \otimes \lambda) = (\mu \otimes \nu) \otimes \lambda$. (See \cite{GanCon2}, Theorem 2.2 and the introduction there.)
 \end{enumerate}
 \end{fact}

\subsection{Definable convolution}\label{sec: def of conv}
Throughout this section, we let $T$ be a first order $\mathcal{L}$-theory expanding a group. We let $\mathcal{G}$ be a sufficiently saturated model of $T$, and $G$ denotes a small elementary submodel. We use letters $x,y$ to denote \emph{singleton} variables, i.e.~ of the sort on which the group is defined. For any formula $\varphi(x,\overline{c}) \in \mathcal{L}_{x}(\mathcal{G})$, we let $\varphi'(x,y,\overline{c}) = \varphi(x \cdot y,\overline{c})$. 

\begin{definition}\label{def: conv} Let $\mu,\nu \in \mathfrak{M}_{x}(\mathcal{G})$, and let $\nu_{y}$ denote the measure in $\mathfrak{M}_{y}(\mathcal{G})$ such that for any $\varphi(y) \in \mathcal{L}_{y}(\mathcal{G})$, $\nu_{y}(\varphi(y)) = \nu(\varphi(x))$.
\begin{enumerate}
	\item We say that \emph{$(\mu,  \nu)$ is $*$-Borel} if for every formula $\varphi(x,\overline{c}) \in \mathcal{L}_{x}(\mathcal{G})$, the triple $(\mu,\nu_{y},\varphi')$ is Borel. We say that \emph{$\mu$ is $*$-Borel} if the pair $(\mu,\mu)$ is $*$-Borel.

\item  If $(\mu,\nu)$ is $*$-Borel, then we define the \textit{(definable) convolution product of $\mu$ with $\nu$} as follows: 
\begin{equation*}
    \mu * \nu (\varphi(x,\overline{c})) = \mu \tilde{\otimes} \nu_{y} (\varphi'(x,y, \overline{c})) =  \int_{\supp(\nu_{y}|_{G})} F_{\mu}^{\varphi'} d\nu_G(y),
\end{equation*}
where  $G$ is some/any small submodel of $\mathcal{G}$ witnessing that  $(\mu,\nu_{y},\varphi')$ is Borel and $\nu_G(y)$ is the Borel measure $\nu_y$ restricted to $\supp(\nu_y|_G)$ (as in Definition \ref{def: gen prod}).
We will routinely write this product simply as $\int F_{\mu}^{\varphi'}d\nu$ when there is no possibility of confusion.
\end{enumerate}
\end{definition}

Note that we are integrating over translates with respect to the \emph{right action of $\mathcal{G}$}, and in general throughout the article, when speaking about $\mathcal{G}$-invariance and related notion, we will typically consider the action of $\mathcal{G}$ \emph{on the right}. This choice is made to make sure that this definition correctly extends Newelski's product of types (Proposition \ref{Newelski}), but of course all of our results hold with respect to left actions modulo obvious modifications. First we check that the convolution operation indeed defines a measure. 

\begin{fact} Let $\mu, \nu \in \mathfrak{M}_{x}(\mathcal{G})$. If $(\mu,\nu)$ is $*$-Borel, then $\mu * \nu$ is a Keisler measure. 
\end{fact}
\begin{proof} Clearly $\mu * \nu ( x = x) =1$ and $\mu * \nu(\neg \varphi(x)) =1 -  \mu * \nu(\varphi(x))$. Assume that $\psi_1(x), \psi_2(x) \in \cL_x(\cU)$ satisfy $\psi_1(x) \wedge \psi_2(x) = \emptyset$. Let  $\theta(x;y) := \psi_1(x \cdot y) \vee \psi_2(x \cdot y)$, and let $G \prec \mathcal{G}$ witness that both $(\mu,\nu_y,\psi'_1)$ and $(\mu,\nu_y,\psi'_2)$ are Borel. Then for any $q \in \supp(\nu|_G)$ and $b \models q$ we have $F_{\mu}^{\theta}(q) = \mu(\theta(x;b)) = \mu(\psi_1(x\cdot b) \vee \psi_2(x \cdot b))$.
As $\psi_1(x) \wedge \psi_2(x) = \emptyset$ implies $\psi_1(x \cdot b) \land \psi_2(x \cdot b) = \emptyset$, we have 
\begin{equation*}
    F_{\mu}^{\theta}(q) = \mu(\psi_1(x\cdot b)) + \mu(\psi_2(x \cdot b)) = F_{\mu}^{\psi'_1}(q) + F_{\mu}^{\psi'_2}(q).
\end{equation*}
Then
\begin{equation*}
    (\mu * \nu)(\psi_1(x) \vee \psi_2(x)) = \int_{\supp(\nu|_{G})} F^{\theta}_\mu d\nu_{G} = \int_{\supp(\nu|_{G})} \left( F^{\psi'_1}_\mu + F^{\psi'_2}_\mu  \right) d\nu_{G}
\end{equation*}
\begin{equation*}
   = \int_{\supp(\nu|_{G})}  F^{\psi'_1}_\mu   d\nu_{G} + \int_{\supp(\nu|_{G})}F^{\psi'_2}_\mu  d\nu_{G}   = (\mu*\nu)(\psi_1(x)) + (\mu*\nu)(\psi_2(x)). 
\end{equation*}
\end{proof}

This notion of convolution extends the notion of the product of invariant types extensively studied by Newelski \cite{N1, N2} and others from the point of view of topological dynamics. The following is easy using Fact \ref{fac: product of inv types}.

\begin{fact}\label{NTM} Let $G \prec \mathcal{G}$ be a small model. Given $p,q \in S_{x}^{\inva}(\mathcal{G},G)$, we define $p * q := \tp(a \cdot b/\mathcal{G}) \in S_{x}^{\inva}(\mathcal{G},G)$, for some/any $(a,b) \models p \otimes q$ in a larger monster model. Then $\left(S^{\inva}_{x}(\mathcal{G},G),* \right)$ is a semigroup, with multiplication continuous in the left coordinate: for each $q \in S_{x}^{\inva}(\mathcal{G},G)$, the map $ -* q: S_{x}^{\inva}(\mathcal{G},G) \to S_{x}^{\inva}(\mathcal{G},G)$ is continuous.
And $(S_{x}(\mathcal{G},G),*)$ is a closed sub-semigroup.
\end{fact}

\begin{proposition}\label{Newelski} Let $ \delta:S_{x}^{\inva}(\mathcal{G},G) \to \mathfrak{M}_{x}^{\inva}(\mathcal{G},G)$ be the map $\delta(p) = \delta_{p}$. Then  $\delta$ is a topological embedding,  and $\delta_{p * q} = \delta_{p} * \delta_{q}$ for any $p,q \in S_x(\mathcal{G},G)$. 
\end{proposition}
\begin{proof}
Clearly $\delta$ is a topological embedding. Now let $\varphi(x) \in \mathcal{L}_{x}(\mathcal{G})$ be arbitrary, then by Proposition \ref{prop: gen prod ext type prod}(3) we have   
\begin{equation*}
\delta_{p * q}(\varphi(x)) = \delta_{p_{x} \otimes q_{y}}(\varphi(x\cdot y)) = \delta_{p_{x}} \tilde{\otimes} \delta_{q_{y}}(\varphi(x \cdot y)) = \delta_{p} * \delta_{q}(\varphi(x)). 
\end{equation*} 
\end{proof}

\begin{remark}\label{rem: ast on types}
	As in Remark \ref{rem: tilde otimes on types}, given $p,q \in S_x(\mathcal{G})$ we say that $(p,q)$ is $\ast$-Borel if $(\delta_p, \delta_q)$ is $\ast$-Borel, and in this case denote by $p \ast q$ the type $r \in S_x(\mathcal{G})$ such that $\delta_r = \delta_p \ast \delta_q$ --- by Proposition \ref{Newelski} this extends the operation on invariant types from Fact \ref{NTM}.
\end{remark}

The next lemma follows by straightforward computations.
\begin{proposition}\label{prop: calc conv} Let $\mu, \mu_1, \ldots, \mu_n, \nu_1, \ldots, \nu_m \in \mathfrak{M}_x(\mathcal{G})$ be arbitrary, and assume that the pairs $(\mu_i, \nu_j)$ are $*$-Borel for all $1 \leq i \leq n, 1\leq j \leq m$. Let $a, b, a_1, \ldots, a_n \in \mathcal{G}$ and $r_1, \ldots, r_n, s_1, \ldots, s_m \in \mathbb{R}_{\geq 0}$ be such that $\sum_{i=1}^{n} r_i = \sum_{j=1}^{m} s_j = 1$. Then:
\begin{enumerate}
   \item $\mu * \delta_e = \delta_{e} * \mu = \mu$,
    \item $\delta_{a} * \delta_{b} = \delta_{ab}$,
    \item $(\delta_{a} * \mu )(\varphi(x)) = \mu(\varphi(a \cdot x))$ for any $\varphi(x) \in \mathcal{L}_x(\cU)$,
    \item $\left(\sum_{i=1}^n r_i \cdot \mu_i \right) * \left(\sum_{j=1}^m s_j \cdot\nu_j \right) = \sum_{i,j = 1}^{n,m}r_i\cdot s_j \cdot (\mu_i * \nu_j)$,
    \item $\left( \left(\sum_{i =1}^n r_i \cdot \delta_{a_i} \right) * \mu  \right)(\varphi(x)) = \sum_{i=1}^n r_i \cdot  \mu \left(\varphi(a_i \cdot x) \right)$ for any $\varphi(x) \in \mathcal{L}_x(\cU)$.
\end{enumerate}
\end{proposition}

Finally, we observe that the following properties of measures are preserved under convolution.
\begin{proposition}\label{prop: pros pres under conv} Let $\mu,\nu \in \mathfrak{M}_{x}(\mathcal{G})$ be such that $(\mu,\nu)$ is $*$-Borel.\begin{enumerate}
\item If $\mu,\nu$ are definable over $G \prec \mathcal{G}$, then $\mu * \nu$ is definable over $G$. 
\item If $\mu,\nu$ are finitely satisfiable over $G \prec \mathcal{G}$, then $\mu * \nu$ is finitely satisfiable over $G$. 
\item If $\mu,\nu$ are finitely approximated over $G \prec \mathcal{G}$, then $\mu*\nu$ is finitely approximated over $G$. 
\item If $\mu(x =b) =0$ for every $b$ in $\mathcal{G}$, then $\mu * \nu( x = b) = 0$ for every $b \in \mathcal{G}$. 
\end{enumerate} 
\end{proposition} 

\begin{proof} Claims (1), (2) and (3) are slight variations on the preservation of the corresponding properties with respect to $\otimes$ (see e.g.~\cite[Lemma
1.6]{NIP3} or \cite[Proposition 2.6]{CG} for (1) and (2), and \cite[Proposition 2.13]{NIP5} or \cite[Proposition 2.10]{CG} for (3)).

(4) Let $b \in \mathcal{G}$ be arbitrary, let $\varphi(x) \in \cL_x(\cU)$ be the formula ``$x=b$'' and let $G \prec \mathcal{G}$  witness that $\left(\mu, \nu_y, \varphi' \right)$ is Borel. Then 
\begin{equation*} \mu * \nu ( x =b ) = \mu \tilde{\otimes} \nu_{y} ( x \cdot y = b) = \int_{\supp(\nu|_{G})} F^{\varphi'}_{\mu} d\nu_{G}(y).
\end{equation*}
And for $q \in S(\nu|_{G})$, $F_{\mu}^{\varphi'}(q) = \mu(x \cdot c = b)$ for some/any $c \models q$ in $\mathcal{G}$. Then, $\mu(x \cdot c = b) = \mu(x =b c^{-1}) = 0$
by assumption. Therefore, $\int F^{\varphi'}_{\mu} d\nu_{G} = \int 0 d\nu_{G} = 0$.
\end{proof} 

\subsection{Idempotent measures}\label{sec: idemp}

We continue working in a theory expanding a group, and begin with some standard definitions.

\begin{definition} Let $\mu \in \mathfrak{M}_x(\mathcal{G})$. 
\begin{enumerate}
\item We say that $\mu$ is \emph{idempotent} if  $\mu$ is $*$-Borel and $\mu*\mu = \mu$.
\item We say that $\mu$ is \emph{right-invariant} if for any formula $\varphi(x) \in \mathcal{L}_x(\mathcal{G})$ and any $a \in \mathcal{G}$, we have $\mu(\varphi(x)) = \mu(\varphi(x\cdot a))$.
\end{enumerate}

\end{definition}
\begin{definition}
	Let $\mathcal{H}$ be a type-definable subgroup of $\mathcal{G}$, where $H(x)$ is the partial type defining the domain of $\mathcal{H}$ (which we associate with the closed set of types implying $H$).
	Then $\mathcal{H}$ is \emph{definably amenable} if there exists a measure $\mu \in \mathfrak{M}_x(\mathcal{G})$ such that $\mu(H(x)) = 1$, and for any $\varphi(x) \in \mathcal{L}_x(\mathcal{G})$ and $a \in \mathcal{H}$ we have $\mu(\varphi(x)) = \mu(\varphi(x \cdot a))$. In this case, we call $\mu$ \emph{right $\mathcal{H}$-invariant}.
\end{definition}

\begin{remark}
	For NIP groups, existence of a right-invariant measure on $\mathcal{H}$ is equivalent to the existence of a left invariant measure on $\mathcal{H}$ (as well as a bi-invariant measure, see \cite[Lemma 6.2]{CS}).
\end{remark}

\begin{proposition}\label{wow} Let $\mathcal{H}$ be a type-definable, definably amenable subgroup of $\mathcal{G}$, defined by a partial type $H(x)$. Suppose that $\mu \in \mathfrak{M}_{x}(\mathcal{G})$ is right $\mathcal{H}$-invariant. Then $\mu$ is idempotent. Moreover, if $\nu$ is another measure such that $\nu(H(x)) = 1$, then $(\mu,\nu)$ is $*$-Borel and $\mu * \nu = \mu$. 
\end{proposition}
\begin{proof}
We show that for any measure $\nu \in \mathfrak{M}_{x}(\mathcal{G})$ such that $\nu(H(x)) = 1$,  $(\mu,\nu)$ is $*$-Borel and $\mu * \nu = \mu$. For ease of notation, we will identify $\nu$ with $\nu_{y}$. Fix a formula $\varphi(x) \in \mathcal{L}_{x}(\mathcal{G})$. Let $G$ be a small elementary submodel of $\mathcal{G}$ containing the parameters of $H(x)$ and $\varphi(x)$. Fix some $q \in \supp(\nu|_{G}) \subseteq S_{y}(G)$, then $q \vdash H(y)$. If not, then $q \in S_{y}(G) \backslash H(y)$. Since $H(y)$ is closed, $S_{y}(G) \backslash H(y)$ is open, hence $S_{y}(G) \backslash H(y) = \bigcup_{i \in I} \psi_i(y)$ for some index set $I$ and $\psi_i \in \mathcal{L}_y(G)$. Then $\psi_i(y) \in q$ for some $i$ and since $q \in \supp(\nu|_{G})$, we know that $\nu(\psi_i(y)) > 0$. But this is a contradiction since $\nu(H(y)) = 1$ and $\psi_i(y)$ is disjoint from $H(y)$. Therefore, if $b \in \mathcal{G}$ and $b \models q$, then $b \in \mathcal{H}$. Now, the function $F_{\mu,G}^{\varphi'}$ is constant on $\supp(\nu|_{G})$ since $F_{\mu,G}^{\varphi'}(q) = \mu(\varphi(x \cdot b)) = \mu(\varphi(x))$ by right $\mathcal{H}$-invariance of $\mu$, hence $(\mu,\nu)$ is $*$-Borel. And $\mu * \nu = \mu$ as
\begin{equation*} \mu * \nu(\varphi(x)) = \int_{\supp(\nu|_{G})} F_{\mu,G}^{\varphi'} d\nu_{G} = \int_{\supp(\nu|_{G})} \mu(\varphi(x)) d\nu_{G} = \mu(\varphi(x)).
\end{equation*} 
In particular, $(\mu,\mu)$ is Borel and $\mu * \mu = \mu$.
\end{proof}

The expectation is that in tame situations, all idempotent measures are of this form for some type-definable subgroup. We will show that this is indeed the case when $\mathcal{G}$ is a stable group in Section \ref{sec: stable groups}, but for now we discuss some examples in which idempotent measures arise. 

If $\mathcal{G}$ is a definably amenable group, and $\mathcal{H}$ is a type-definable subgroup of finite index (hence definable), then $\mathcal{H}$ is definably amenable (if $\mu$ is a right-invariant measure on $\mathcal{G}$, then $\mu_{\mathcal{H}}(\varphi(x)) := [\mathcal{G} : \mathcal{H}] \cdot \mu(\varphi(x) \cap H(x))$ gives a right-invariant measure on $\mathcal{H}$). This generalizes to type-definable subgroups of bounded index when $\mathcal{G}$ is NIP.

\begin{proposition}\label{prop: bdd index def am NIP} Assume that $\mathcal{G}$ is definably amenable and NIP, and let $\mathcal{H}$ be a type-definable subgroup of $\mathcal{G}$ of bounded index. Then $\mathcal{H}$ is also definably amenable.
\end{proposition}

Proposition \ref{prop: bdd index def am NIP} follows from a slightly generalized construction of the $\mathcal{G}$-invariant measures $\mu_p$ from \cite{NIP1, NIP2, CS} in NIP groups. We will use some properties of the absolute type-definable connected component $\mathcal{G}^{00}$, the intersection of all type-definable subgroups of $\mathcal{G}$ of bounded index, and refer to the aforementioned texts for further details. To be compatible with our set up for convolutions, we work with \emph{$\mathcal{G}$ acting on the right}.
By NIP, $\mathcal{G}^{00}$ is a normal subgroup type-definable over $\emptyset$, and let $G \prec \mathcal{G}$ be an arbitrary small model. As usual,  $\pi: \mathcal{G} \to \mathcal{G}/\mathcal{G}^{00}$ is the surjective group homomorphism with $\pi(g)$ depending only on $\tp(g/G)$. Then $\mathcal{G}/\mathcal{G}^{00}$ is a compact Hausdorff topological group with respect to the logic topology,  i.e.~a subset $X$ of $\mathcal{G}/\mathcal{G}^{00}$ is closed if and only if $\pi^{-1}(X)$ is type-definable, if and only if $\pi^{-1}(X)$ is type-definable over $G$. The induced map $S_x(G) \to \mathcal{G}/\mathcal{G}^{00}$ is continuous. With respect to this topology, closed subgroups of $\mathcal{G}/\mathcal{G}^{00}$ are in a bijective correspondence with type-definable subgroups of $\mathcal{G}$ of bounded index (equivalently, containing $\mathcal{G}^{00}$). Namely, if $K$ is a closed subgroup of $\mathcal{G}/\mathcal{G}^{00}$, then $\mathcal{H} := \pi^{-1}(K)$ is a type-definable set containing $\mathcal{G}^{00} = \pi^{-1}(e_K)$, and is a group since $\pi$ is a group homomorphism (and vice versa). Also, if $\mathcal{H} \subseteq \mathcal{G}$ is type-definable, then $K := \pi(\mathcal{H}) \subseteq \mathcal{G}/\mathcal{G}^{00}$ is a closed subgroup (as $\pi:S_x(G) \to \mathcal{G}/\mathcal{G}^{00}$ is a closed map).

Recall that a global type $p \in S_x(\mathcal{G})$ is \emph{strongly $f$-generic} over $G$ if $p \cdot g$ is $G$-invariant for every $g \in \mathcal{G}$. If $\mathcal{G}$ is definably amenable and $G$ is an arbitrary small model, then there exists a type $p$ strongly $f$-generic over $G$ (see \cite{NIP2}). Moreover, as every right translate of a strong $f$-generic over $G$ is again a strong $f$-generic over $G$, we may assume $p(x) \vdash \mathcal{G}^{00}(x)$.
%
%

\begin{proof}[Proof of Proposition \ref{prop: bdd index def am NIP}]
Let $K := \pi(\mathcal{H})$, then  $\pi^{-1}(K) = \mathcal{H}$ (by the fourth isomorphism theorem for groups), hence $K$ is a closed subgroup of $\mathcal{G}/\mathcal{G}^{00}$. Denote by $\nu$ the right-invariant Haar measure on Borel subsets of $K$ normalized by $\nu(K) = 1$.

Let $p \in S_x(\mathcal{G})$ be a strong $f$-generic over $G$ with $p(x) \vdash \mathcal{G}^{00}(x)$, so in particular $p \vdash \mathcal{H}$ and $p \cdot g = p$ for every $g \in \mathcal{G}^{00}$.
For a formula $\varphi(x)\in \mathcal{L}_x(\mathcal{G})$, let 
	$$A_{\varphi,p} := \left\{ \bar{g} \in K : \varphi(x) \in  p \cdot \bar{g} \right\}.$$
	
	Then $A_{\varphi,p}$ is a Borel subset of $K$ (as  $A_{\varphi,p} = K \cap \left\{ \bar{g} \in  \mathcal{G}/\mathcal{G}^{00}: \varphi(x) \in  p \cdot \bar{g} \right\}$, and the latter set is Borel by \cite[Proposition 5.6]{NIP2}).
	We define
	 $$\mu_{p,\nu}(\varphi(x)) := \nu(A_{\varphi, p}).$$
	 Then we have the following.
	 \begin{itemize}
	 	\item $\mu_{p,\nu}$ is a Keisler measure with $\mu_{p,\nu}(H) = 1$.
	 	
	 	It is easy to check that $\mu_{p,\nu}$ is a measure. And by regularity, $\mu_{p,\nu}(H) = \inf\{ \mu_{p,\nu}(\psi(x)) : H
	 	(x) \vdash \psi(x), \psi(x) \in \mathcal{L}_x(\mathcal{G}) \}$, and as $p \vdash H \vdash \psi$ for all such $\psi$, we have that $A_{\psi,p} = K$ by definition, hence $\mu_{p,\nu}(\psi) = 1$.
	 	\item $\mu_{p,\nu}$ is right $\mathcal{H}$-invariant (as $\mu_{p,\nu}(\varphi(x)\cdot g) = \nu(A_{\varphi \cdot g,p}) = \nu(A_{\varphi,p} \cdot \pi(g)) = \nu(A_{\varphi,p}) = \mu_{p,\nu}(\varphi(x))$ by right $K$-invariance of $\nu$, as $\pi(g) \in K$).
	 \end{itemize}
	 Hence $\mathcal{H}$ is definably amenable, witnessed by $\mu_{p,\nu}$.
\end{proof}

\begin{question}
Is Proposition \ref{prop: bdd index def am NIP}	true without the NIP assumption?
\end{question}

%

Classification of measures supported on finite subsets of $\mathcal{G}$ follows  from Wendel's theorem.

\begin{proposition} \label{prop: finite realized support} If $\mu$ is a measure on $\mathcal{G}$ whose support is a finite collection of realized types, then $\mu$ is idempotent if and only if $\mu = \frac{1}{|\mathcal{H}|}\sum_{a \in \mathcal{H}} \delta_{a}$ for some finite subgroup $\mathcal{H}$ of $\mathcal{G}$. 
\end{proposition}
\begin{proof} $(\Leftarrow)$ is by Proposition \ref{wow}. 

$(\Rightarrow)$ Assume that $\supp(\mu) = \{a_1,...,a_n\} = A \subseteq \mathcal{G}$. As $\mu$ is idempotent, $\supp(\mu)$ is closed under multiplication (if not, then there exists $c \in \mathcal{G} \setminus A$ such that $c = a_i \cdot a_j$ and $c$ for some $i,j$; then $\mu(x = c) = 0$, but $\mu * \mu(x = c) >0$). Therefore $A$ is closed under products. As any finite subset of a group closed under products is a subgroup, $A$ is a compact group, and $\mu|_{A}$ is an idempotent measure on $A$. Therefore, by \cite[Theorem 1]{Wendel}, $\mu|_{A}$ is the unique Haar measure on the subgroup $\supp(\mu|_A)$ of $A$. But as $\supp(\mu) = A$, we conclude that $\mu = \frac{1}{n} \sum_{a \in A} \delta_{a}$. 
\end{proof} 

Finally, we observe a sufficient condition for idempotence to be preserved under convolution  in the NIP context.
\begin{proposition}\label{prop: commute preserves idemp} If $\mathcal{G}$ is NIP and abelian, and both $\mu,\nu$ are idempotent and dfs, then $\mu * \nu$ is idempotent and dfs.
\end{proposition}
\begin{proof}
	Fix a formula $\varphi(x) \in \mathcal{L}_x(\mathcal{G})$ and assume that $G \prec \mathcal{G}$ witnesses that both  $(\mu,\nu, \varphi)$ and $(\nu,\mu, \varphi)$ are Borel, and both $\mu$ and $\nu$ are dfs over $G$ (taking a common extension of the models witnessing each of this properties separately). By Proposition \ref{prop: pros pres under conv}, $\mu \ast \nu$ is dfs over $G$.  By Fact \ref{fac: meas commute}(3), $\mu$ and $\nu$ commute, so we have
\begin{equation*}
    \mu * \nu (\varphi(x)) = \mu_{x} \otimes \nu_{y} (\varphi(x \cdot y)) = \nu_{y} \otimes \mu_{x}(\varphi(x \cdot y)). 
\end{equation*}
By change of variables and abelianity of $\mathcal{G}$, we can conclude
\begin{equation*}
    = \nu_{x} \otimes \mu_{y}(\varphi(y \cdot x)) = \nu_{x} \otimes \mu_{y}(\varphi(x \cdot y)) = \nu * \mu(\varphi(x)). 
\end{equation*}
Now, let $\lambda := \mu * \nu$. Using associativity of $*$ in the NIP context (see Proposition \ref{NIP:measure}), 
\begin{equation*} \lambda * \lambda = \mu * \nu * \mu * \nu = \mu * \mu * \nu * \nu = \mu * \nu = \lambda. 
\end{equation*}
\end{proof}

\section{Supports of idempotent measures}\label{sec: supports}
In this section, we will show (in an arbitrary theory) that if $\mu$ is definable, invariantly supported (see Definition \ref{def: inv sup meas}) and idempotent, then $(\supp(\mu),*)$ is a compact, left-continuous  semigroup with no closed two-sided ideals. The assumption ``definable and invariantly supported'' is satisfied when $\mu$ is a dfs measure in an arbitrary theory (by Lemma \ref{lem: sup on inv types}(1)), and when $\mu$ is an arbitrary definable measure in an NIP theory (by Lemma \ref{lem: sup on inv types}(3)).

We begin by considering two examples, which illustrate in particular that the support of an idempotent dfs Keisler measure need not be a group in general.

\begin{example}\label{exa: 1} Let $T = T_{\doag}$ be the complete theory of a divisible ordered abelian group in the language $\{+,<,0,1\}$. Let $\mathcal{G}$ be a monster model of $T$ and consider $G := \mathbb{Q}$ as an elementary substructure in the natural way. Let $p_{\infty}$ be the unique global type finitely satisfiable in $G$ and extending $\{x > a: a \in \mathbb{Q}\}$. Let $p_{-\infty}$ be the unique global type finitely satisfiable in $G$ and extending $\{ x <a :a \in \mathbb{Q}\}$. 
Let $\mu := \frac{1}{2}\delta_{p_{- \infty}} + \frac{1}{2}\delta_{p_{\infty}}$, we claim that $\mu, \delta_{p_{\infty}},\delta_{p_{-\infty}} \in \mathfrak{M}_x(\mathcal{G})$ are idempotent. By Proposition \ref{prop: pros pres under conv}, the product $\delta_{\alpha} * \delta_{\beta}$ for $\alpha,\beta \in \{p_{\infty},p_{-\infty}\}$ is finitely satisfiable in $\mathbb{Q}$. Then, using Proposition \ref{prop: calc conv}, it is not hard to verify the following calculation:
\begin{equation*}
    \mu * \mu = \Big(\frac{1}{2}\delta_{p_{- \infty}} + \frac{1}{2}\delta_{p_{\infty}}\Big) * \Big(\frac{1}{2}\delta_{p_{- \infty}} + \frac{1}{2}\delta_{p_{\infty}}\Big) 
\end{equation*}
\begin{equation*}
    =\frac{1}{4} \Big(\delta_{p_{-\infty} }* \delta_{ p_{-\infty}} \Big) + \frac{1}{4} \Big(\delta_{p_{- \infty}} * 
    \delta_{p_{\infty}}\Big) + 
    \frac{1}{4} \Big(\delta_{p_{\infty}} * \delta_{p_{-\infty}}\Big) + 
    \frac{1}{4}\Big(\delta_{p_{\infty}} * \delta_{p_{\infty}}\Big)
\end{equation*}
\begin{equation*}
   = \frac{1}{4}\delta_{p_{- \infty}} + \frac{1}{4}\delta_{p_{\infty}} + \frac{1}{4} \delta_{p_{- \infty}} + \frac{1}{4} \delta_{p_{\infty}} = \frac{1}{2}\delta_{p_{-\infty}} + \frac{1}{2}\delta_{p_{\infty}} = \mu.
\end{equation*}
We observe that while $(\supp(\delta_{p_{\infty}}),*)$ and $(\supp(\delta_{p_{-\infty}}),*)$ are groups (with a single element), $(\supp(\mu),*)$ is not a group since it does not contain an identity element. 
\end{example} 

\begin{example}\label{circle}
Let $G = (S^{1},\cdot,C(x,y,z))$ be the standard circle group over $\mathbb{R}$, with $C$ the cyclic clockwise ordering. Let $T_{O}$ be the corresponding theory. Let $\mu$ be the Keisler measure on this structure which corresponds to the restriction of the Haar measure on $S^{1}$. Let $\mathcal{G}$ be a monster model of $T_{O}$ such that $S^{1} \prec \mathcal{G}$. Then $\mu$ is smooth over $S^{1}$ and admits a unique global extension $\tilde{\mu}$. We remark that $\tilde{\mu}$ is right invariant, hence idempotent (Proposition \ref{wow}). Let $\textrm{st}:S_{x}(\mathcal{\mathcal{G}}) \to S^{1}$ be the standard part map. Assume that $p \in \supp(\tilde{\mu})$ and $\textrm{st}(p) = a$. Then $\varphi_\varepsilon(x) := C(a-\varepsilon, x, a + \varepsilon) \notin p$ for every infinitesimal $\varepsilon \in \mathcal{G}$ ($x \neq a \in p$ as $\mu(x=a) = 0$, and if $\varphi_\varepsilon(x) \in p$, then $\tilde{\mu}(\varphi_\varepsilon(x) \land x \neq a) > 0$, but $\varphi_\varepsilon (G) = \{ a \}$ --- contradicting finite satisfiability of $\tilde{\mu}$ in $G$). As the types are determined by the cuts in the circular order, it follows that for every $a \in S^1$ there are exactly two types $a_{+}(x),a_{-}(x) \in \supp(\tilde{\mu})$ determined by whether $C(a+\varepsilon, x, b)$ holds for every infinitesimal $\varepsilon$ and $b \in G$, or $C(b, x, a-\varepsilon)$
holds for every infinitesimal $\varepsilon$ and $b \in G$, respectively.
It follows that $(\supp(\tilde{\mu}),*) \cong S^{1} \times \{+,-\}$ with multiplication defined by:
\begin{equation*} a_{\delta} * b_{\gamma} = (a \cdot b)_{\delta}
\end{equation*}
for all $a,b \in S^1$ and $\delta,\gamma \in \{+,-\}$. Again, $(\supp(\mu),*)$ is not a group. 
\end{example}

Next we establish various properties of $(\supp(\mu),*)$ when $\mu$ is a global  idempotent measure which is definable and invariantly supported. Given $S_1, S_2 \subseteq S_x(\mathcal{G})$, we write $S_1 * S_2 := \{ p_1 * p_2 \in S_x(\mathcal{G}) : p_i \in S_i \}$ (under the assumption that all such products are defined, i.e.~assuming $(p_1, p_2)$ is $\ast$-Borel for all  $p_i \in S_i$ --- see Remark \ref{rem: ast on types}).
The assumption of being invariantly supported in the lemmas below is only needed to ensure that $\supp(\mu)*\supp(\mu)$ is defined (Fact \ref{NTM}).

\begin{proposition}\label{semigroup} Let $\mu,\nu \in \mathfrak{M}_{x}(\mathcal{G})$. Assume that $\mu$ is definable, and both $\mu$ and $\nu$ are invariantly supported. Then:
\begin{enumerate}
	\item $\supp(\mu) * \supp(\nu) \subseteq \supp(\mu * \nu)$;
	\item $\supp(\mu) * \supp(\nu)$ is a dense subset of $\supp(\mu * \nu)$.
\end{enumerate}
\end{proposition} 

\begin{proof} (1) Assume that $p \in \supp(\mu), q \in \supp(\nu)$, and let $\varphi(x) \in p *q$. Choose $G \prec \mathcal{G}$ such that $\mu$ is definable over $G$, $p,q$ are finitely satisfiable in $G$, and $G$ contains all the parameters from $\varphi$. We need to show that $\mu * \nu(\varphi(x)) > 0$. Now, 
\begin{equation*}
    \mu * \nu(\varphi(x)) = \int_{\supp(\nu|_{G})} F_{\mu, G}^{\varphi'}d\nu_{G}
\end{equation*} Since $\mu$ is definable, the map $F_{\mu,G}^{\varphi'}:\supp(\nu|_{G}) \to [0,1]$ is continuous. Therefore, it suffices to find some $r \in \supp(\nu|_{G})$ such that $F_{\mu,G}^{\varphi'}(r) > 0$. Consider $r := q|_{G}$. Then, $F_{\mu,G}^{\varphi'}(q|_{G}) = \mu(\varphi(x\cdot b))$, where $b \models q|_{G}$. Then, $\varphi(x \cdot b) \in p$ and since $p \in \supp(\mu)$, we have that $\mu(\varphi(x\cdot b)) > 0$. Hence, $F_{\mu}^{\varphi'}(q|_{G}) > 0$ and so $\mu * \nu(\varphi(x)) > 0$. 

(2) By (1), we already know that $\supp(\mu) * \supp(\nu) \subseteq \supp( \mu * \nu)$. Fix some $r \in \supp(\mu * \nu)$ and a formula $\varphi(x) \in r$.  We need to find $p \in \supp(\mu)$ and $q \in \supp(\nu)$ such that $\varphi(x) \in p*q$. Choose $G$ such that $\mu$ is definable over $G$, all types in $\supp(\mu), \supp(\nu)$ are invariant over  $G$, and $G$ contains the parameters of $\varphi(x)$. Since $\varphi(x) \in r$ and $r$ is in the support of $\mu * \nu$, we know that $\mu * \nu (\varphi(x)) >  0$. Therefore, $\int_{\supp(\nu|_G)} F_{\mu, G}^{\varphi'} d(\nu_{G})>0$, and so there exists some $t \in \supp(\nu|_{G})$ such that $F_{\mu, G}^{\varphi'}(t) > 0$. If $c \models t$, then $\mu(\varphi(x \cdot c)) > 0$. So, by Proposition \ref{prop: supp 1}(1), there exists $p \in \supp(\mu)$ such that $\varphi(x \cdot c) \in p$. By Proposition \ref{coheir}, we let $q \in \supp(\nu)$ be such that $q|_{G} = t$. By construction, we then observe that $\varphi(x) \in p * q$. 
\end{proof}

\begin{corollary}\label{qcont} Assume that $\mu$ is definable, invariantly supported and idempotent. Then $\left(\supp(\mu),* \right)$ is a compact Hausdorff (with the subspace topology) semigroup which is left-continuous, i.e.~the map $-*q: \supp(\mu) \to \supp(\mu)$ is continuous for each $q \in \supp(\mu)$.
\end{corollary}

\begin{proof} By Proposition \ref{prop: supp 1}(2), $\supp(\mu)$ is a compact Hausdorff space.  By Proposition \ref{semigroup}, $\supp(\mu) * \supp(\mu) \subseteq \supp(\mu * \mu) = \supp(\mu)$. Now, choose some $G \prec \mathcal{G}$ such that $\mu$ is definable over $G$, and all types in $\supp(\mu)$ are invariant over $G$. Then $(\supp(\mu),*)$ is a sub-semigroup of $(S_{x}^{\inva}(\mathcal{G},G),*)$ and $*$ is left-continuous by Fact \ref{NTM}.
\end{proof}

We now define some global functions which mimic the map $y \mapsto \int f(x \cdot y)d\mu(x)$.

\begin{definition} Let $\mu \in \mathfrak{M}_x(\mathcal{G})$ be definable, and fix $\varphi(x) \in \mathcal{L}_{x}(\mathcal{G})$. We then define the global function $D_{\mu}^{\varphi'}:S_{y}(\mathcal{G}) \to [0,1]$ via $p \mapsto \mu(\varphi(x\cdot c))$, for some/any $c \models p|_{G}$ and small $G \prec \mathcal{G}$ containing the parameters of $\varphi(x)$ and such that $\mu$ is definable over $G$.
\end{definition}

%

Note that for any formula $\varphi(x) \in \mathcal{L}_{x}(\mathcal{G})$, the map $D^{\varphi'}_{\mu}$ is continuous: $D_{\mu}^{\varphi'} = F_{\mu,G}^{\varphi'}\circ r$, where $r: S_y(\mathcal{G}) \to S_y(G)$ is the restriction map, and $F_{\mu,G}^{\varphi'}$ is continuous by definability of $\mu$. The next two results are adapted from Glicksberg's work on semi-topological semigroups into the general model theory context. In particular, see \cite{Glicksberg1,Glicksberg2}.

\begin{proposition} \label{prop: max attained val} Let $\mu \in \mathfrak{M}_{x}(\mathcal{G})$  be definable, invariantly supported and idempotent, and $\varphi(x) \in \mathcal{L}_{x}(\mathcal{G})$ arbitrary. Assume that $D_{\mu}^{\varphi'}|_{\supp(\mu)}$ attains a maximum at $q \in \supp(\mu)$ (exists as this is a continuous function on a compact set). Then for any $p \in \supp(\mu)$, we have that $D_{\mu}^{\varphi'}(q) = D_{\mu}^{\varphi'}(p * q)$.
\end{proposition} 

\begin{proof} Fix a small model $G_{0} \prec \mathcal{G}$ such that $\mu$ is definable over $G_{0}$, and $G_0$ contains the parameters of $\varphi(x)$. Let $b \models q|_{G_{0}}$ and let $\theta(x;y) := \varphi((x\cdot y) \cdot b)$. Now fix a larger submodel $G \prec \mathcal{G}$ such that $G_0b \subset G$. Let $\delta : = \mu(\varphi(x\cdot b))$. Observe that then for any $t \in \supp(\mu|_{G})$, $a \models t$, and $\tilde{t} \in \supp(\mu)$ such that $\tilde{t}|_{G} = t$, we have $F_{\mu, G}^{\theta}(t) = \mu(\varphi(x \cdot a) \cdot b) = \mu \left(\varphi(x \cdot (a \cdot b)) \right) = D_{\mu}^{\varphi'}(\tilde{t}*q) \leq D_{\mu}^{\varphi'}(q)= \delta$ (by the assumption on $q$). We conclude that for any $t \in \supp(\mu|_{G})$, $F_{\mu,G}^{\theta}(t) \leq \delta$. On the other hand,
\begin{equation*} \delta =  D_{\mu}^{\varphi'}(q) = \mu(\varphi(x \cdot b)) = \mu * \mu (\varphi(x \cdot b)) = \mu_{x} \tilde{\otimes} \mu_{y}(\theta(x;y))
\end{equation*}
\begin{equation*} = \int_{\supp(\mu|_G)} F_{\mu,G}^{\theta} d\mu_G.
\end{equation*}
 Therefore, $F_{\mu}^{\theta} = \delta$ almost everywhere (with respect to $\mu_{G}$). Since both maps are continuous, they are equal on $\supp(\mu|_{G})$. Finally, for any $p \in \supp(\mu)$ and $a \models p$, we have:
\begin{equation*} D_{\mu}^{\varphi'}(q)=\delta = F_{\mu,G}^{\theta}(p|_{G}) = \mu(\varphi((x \cdot a) \cdot b)) = \mu(\varphi(x \cdot (a \cdot b))) = D_{\mu}^{\varphi'}(p * q),
\end{equation*}
as wanted.
\end{proof}

\begin{theorem}\label{two} Let $\mu \in \mathfrak{M}_{x}(\mathcal{G})$ be definable, invariantly supported and idempotent. Let $I \subset \supp(\mu)$ be a closed two-sided ideal. Then, $I = \supp(\mu)$.  
\end{theorem}

\begin{proof} If $I$ is dense in $\supp(\mu)$, then $I = \supp(\mu)$. So we may assume that $I$ is not dense in $\supp(\mu)$. Therefore, there exists  some $\varphi(x) \in \mathcal{L}_{x}(\mathcal{G})$ such that $\varphi(x) \cap \supp(\mu) \neq \emptyset$ and  $\varphi(x) \cap  I = \emptyset$. Let $G \prec \mathcal{G}$ be a small model containing the parameters of $\varphi$, and such that $\mu$ is definable and invariantly supported over $G$.

\begin{cla}
	There exists some $q \in \supp(\mu)$ such that $D_{\mu}^{\varphi'}(q)>0$.
\end{cla}
\begin{proof}
Assume not. 
Let $p,q \in \supp(\mu)$ be arbitrary. Let $b \models q|_G, a \models p|_{Gb}$. Then $\mu(\varphi(x \cdot b)) = D^{\varphi'}_{\mu}(q) = 0 $ by assumption, hence $\models \neg \varphi(a \cdot b)$ as $p \in \supp(\mu)$, so $\varphi(x) \notin p * q$.

Consider now the continuous characteristic function $\chi_{\varphi}:\supp(\mu) \to \{0,1\}$. By Proposition \ref{semigroup}(2) and the previous paragraph, $\chi_{\varphi}$ vanishes on a dense subset $\supp(\mu) * \supp(\mu)$ of $\supp(\mu)$, hence $\chi_{\varphi}$ vanishes on $\supp(\mu)$. But this contradicts the choice of $\varphi$.
\end{proof}
So there exists some $q \in \supp(\mu)$ such that $D_{\mu}^{\varphi'}(q) > 0$. Then, since $D_{\mu}^{\varphi'}$ is continuous, it attains a maximum $\delta >0$ on some $r \in \supp(\mu)$. 

\begin{cla}
	For any $h \in I$, we have $D_{\mu}^{\varphi'}(h) = 0$.
\end{cla}
\begin{proof}
Let $h \in I$. Then $D_{\mu}^{\varphi'}(h) = \mu(\varphi(x \cdot b))$, where $b \models h|_{G}$. Then
\begin{equation*}\mu(\varphi(x \cdot b)) = \mu \left(\{p \in \supp(\mu): \varphi(x \cdot b ) \in p\} \right) = \mu \left(\{ p \in \supp(\mu) : \varphi(x) \in p*h\} \right).
\end{equation*}
As  $I$ is a left ideal, we have $\supp(\mu) * h \subseteq I$. By assumption, $\varphi(x) \cap I = \emptyset$, and so we have $\{p \in \supp(\mu) : \varphi(x) \in p * h\} = \emptyset$. Therefore, $D_{\mu}^{\varphi'}(h) = 0$. 
\end{proof}
 
Finally, since $I$ is a right ideal, we have that $h*r \in I$. Therefore, using Proposition \ref{prop: max attained val} and the claim,
\begin{equation*} 0 < D_{\mu}^{\varphi'}(r) = D_{\mu}^{\varphi'}(h * r) = 0.
\end{equation*} 
Therefore, we obtain a contradiction.  
\end{proof}

\begin{corollary} Assume that $|\supp(\mu)| >1$, i.e.~$\mu$ is not a type. Then  $\supp(\mu)$ contains no zero elements, i.e.~there is no element $p \in \supp(\mu)$ such that for any $q$ in $\supp(\mu)$, $p * q = q*p = p$. 
\end{corollary} 

\begin{proof} If $p$ is a zero-element, then $\{p\}$ is a closed two sided ideal.
\end{proof}

We make some further observations on the structure of the semigroup $\supp(\mu)$ under the additional assumptions on the idempotent measure $\mu$. 
We recall the following structural theorem of Ellis (with the roles of multiplication on the left and on the right exchanged everywhere).

\begin{fact}\label{Ellis}\cite[Proposition 4.2]{Ellis} Assume that $(S,\cdot)$ is a compact Hausdorff semigroup which is left-continuous (i.e.~such that for any $a \in S$, the map $- \cdot a:S \to S$ is continuous). Then, there exists a minimal left ideal $I$ (which is automatically closed). We let $J(I) = \{i \in I: i^{2} = i\}$ be the set of idempotents in $I$.
\begin{enumerate}
\item $J(I)$ is non-empty. 
\item For every $p \in I$ and $i \in J(I)$, we have that $p \cdot i = p$.
\item $I = \bigcup \{i \cdot I: i \in J(I)\}$, where the union is over disjoint sets, and each set $i \cdot I$ is a group with identity $i$.
\item $I \cdot q$ is a minimal left ideal for all $q \in S$. 
\end{enumerate}
\end{fact} 

\noindent Assume that $\mu \in \mathfrak{M}_x(\mathcal{G})$ is definable, invariantly supported and idempotent. Then $(\supp(\mu), *)$ is a semigroup satisfying the assumption of Fact \ref{Ellis} by Corollary \ref{qcont}.
\begin{definition} 
We let $I_{\mu}$ denote a minimal (closed) left ideal of $(\supp(\mu), *)$ (it exits by Fact \ref{Ellis}).
We say that $\mu$ is \emph{minimal} if $I_{\mu} = \supp(\mu)$.
\end{definition}


In particular, if $\mu$ is minimal, then $\supp(\mu)$ is a disjoint union of subgroups. 

\begin{example} For example, the measure $\tilde{\mu}$ considered in Example \ref{circle} is minimal.
\end{example}
 
\begin{proposition} \label{prop: D is const on supp} Assume that $\mu \in \mathfrak{M}_x(\mathcal{G})$ is definable, invariantly supported, idempotent and minimal (i.e. $I_\mu = \supp(\mu)$). Let $\varphi(x) \in \mathcal{L}_{x}(\mathcal{G})$ be any formula. Then for any $p,q \in \supp(\mu)$, we have that $D_{\mu}^{\varphi'}(p) = D_{\mu}^{\varphi'}(q)$.
\end{proposition} 
\begin{proof} By Fact \ref{Ellis}, $\supp(\mu)= \bigcup \{i * \supp(\mu): i \in J(I_\mu)\}$. By continuity, $D_{\mu}^{\varphi'}$ attains a maximum at some $p \in \supp(\mu)$. Let now $q \in \supp(\mu) = I_\mu$ be arbitrary. Then $q \in i *I_\mu$ for some $i \in J(I_\mu)$. Also $i * p \in i * I_\mu$ as $I_\mu = \supp(\mu)$. As $i*I_\mu$ is a group by Fact \ref{Ellis}(3), there exists some $r \in i * I_\mu$ such that $r * (i  * p) = q$. 
But then, applying Proposition \ref{prop: max attained val}, we have
$$D_{\mu}^{\varphi'}(p) = D_{\mu}^{\varphi'}((r * i) * p) = D_{\mu}^{\varphi'}(r * (i * p)) = D_{\mu}^{\varphi'}(q).$$
As $q \in \supp(\mu)$ was arbitrary, this shows the proposition.
%
%
%
%
%
%
\end{proof} 

\begin{proposition} Assume that $\mu \in \mathfrak{M}_x(\mathcal{G})$ is definable, invariantly supported, idempotent and minimal. Then for every $\varphi(x) \in \mathcal{L}_{x}(\mathcal{G})$, $\mu(\varphi(x)) = D_{\mu}^{\varphi'}(p)$ for any $p \in \supp(\mu)$. 
\end{proposition} 

\begin{proof} Assume not. By Proposition \ref{prop: D is const on supp} and replacing $\varphi(x)$ by $\neg \varphi(x)$ if necessary, we may assume that $\mu(\varphi(x)) > D_{\mu}^{\varphi'}(i)$, where $i$ is an idempotent in $\supp(\mu)$. Then $\mu \left(\varphi(x) \wedge \neg \varphi(x \cdot b) \right) > 0$, where $b \models i|_{G}$ and $G \prec \mathcal{G}$ is chosen as usual. Hence there exists $q \in \supp(\mu)$ such that $\varphi(x) \wedge \neg \varphi(x \cdot b) \in q$. Then $\varphi(x) \in q$, and $\neg \varphi(x) \in q * i$. However, $q * i = q$ by Fact \ref{Ellis}(2), and so we have $\varphi(x), \neg \varphi(x) \in q$  ---  a contradiction.
\end{proof}
\noindent A direct translation of the previous proposition then says that minimal idempotent measures are ``generically'' right-invariant on their supports. 

\begin{corollary} \label{cor: meas on sup is invariant}Assume that $\mu \in \mathfrak{M}_x(\mathcal{G})$ is definable, invariantly supported, idempotent and minimal. Let $\varphi(x;\overline{b}) \in \mathcal{L}_{x}(\mathcal{G})$. Then, for any $a \in \mathcal{G}$ such that $tp(a/G\overline{b}) \in \supp(\mu|_{G\overline{b}})$, we have $\mu(\varphi(x)) = \mu(\varphi(x \cdot a))$.

\end{corollary} 

\noindent Finally, we record a corollary for the case when the group $\mathcal{G}$ is stable and abelian.

%

\begin{remark} $I_{\mu} = \supp(\mu)$ if and only if for every $p,q$ in the $\supp(\mu)$ there exists $r \in \supp(\mu)$ such that $r * q =p$. 
\end{remark} 

\noindent The following corollary is a direct consequence of Glicksberg's theorem for semi-topological semigroups \cite{Glicksberg2} (note that unless the group is stable and abelian, we only have continuity of $*$ on the left, so we were not in the context of Glicksberg's theorem in the earlier considerations).

\begin{corollary}\label{cor: supp stab ab}
	If $\mathcal{G}$ is stable, abelian and $\mu \in \mathfrak{M}_x(\mathcal{G})$ is idempotent, then $\supp(\mu)$ is an abelian compact Hausdorff topological group.
\end{corollary}
\begin{proof}

Note that $\mu$ is automatically dfs  by Fact \ref{fac: props of measures relation}(3), hence the results of this section apply to it. We see that $(\supp(\mu), *)$ is commutative, as in  Proposition \ref{prop: commute preserves idemp}. Then $*$ is both left and right-continuous.
Hence $I_\mu = \supp(\mu)$ by Theorem \ref{two}. But this is equivalent to: for every $p,q \in \supp(\mu)$ there exists $r \in \supp(\mu)$ such that $r * q =p$. By commutativity of $*$ and Fact  \ref{Ellis}, this implies that $S(\mu)$ is a group. Finally, by a classical theorem of Ellis \cite{ellis1957locally}, separate continuity of multiplication implies joint continuity for (locally) compact groups.
\end{proof}

\noindent Using this corollary, we can quickly describe idempotent measures in strongly minimal groups.

\begin{example}\label{exa: str min}
	Let $\mathcal{G}$ be a strongly minimal group. Then the idempotent measures are precisely of the following form:
	\begin{enumerate}
		\item Haar measures on finite subgroups of $\mathcal{G}$;
		\item $\delta_{p}$, where $p$ is the unique non-algebraic type in $S_{x}(\mathcal{G})$.
	\end{enumerate}

\end{example}
\begin{proof}
Assume that $\mathcal{G}$ is strongly minimal, then it is abelian, and let $\mu$ be an idempotent measure. As $\mathcal{G}$ is in particular $\omega$-stable, by Fact \ref{fac: props of measures relation}(3c) $\mu = \sum_{i \in  \omega} r_i \cdot p_i$ for some $p_i \in S_x(\mathcal{G})$ and some $r_i \in \mathbb{R}_{\geq 0}$ with $\sum_{i \in \omega} r_i = 1$. By strong minimality, let $p \in S_x(\mathcal{G})$ be the unique non-algebraic type. Then clearly $\supp(\mu) = \cl(\{p_i : i \in \omega \})  \subseteq \{p_i : i \in \omega \} \cup  \{p\}$, in particular $\supp(\mu)$ is countable. By Corollary \ref{cor: supp stab ab}, $\supp(\mu)$ is a compact group, and every countable compact group must be finite (using the existence of finite Haar measure). 
If $p \notin \supp(\mu)$, then $\supp(\mu)$ is a finite subgroup of $\mathcal{G}$, and we are in the first case by Proposition \ref{prop: finite realized support}. Assume that $p \in \supp(\mu)$. Note that $p$ is clearly right $\mathcal{G}$-invariant, hence $p \ast q = p$ for any $q \in \supp(\mu)$ by Proposition \ref{wow}. As $(\supp(\mu), \ast)$ is a group, this implies $\supp(\mu) = \{ p \}$.
%
%
\end{proof}
This example is generalized to arbitrary stable groups in the next section.

\section{Idempotent measures in stable groups}\label{sec: stable groups}
	In this section we classify idempotent measures on a stable group, demonstrating that they are precisely the invariant measures on its type-definable subgroups. Our proof relies on the results of the previous section and a variant of Hrushovski's group chunk theorem due to Newelski \cite{N3}. We will assume some familiarity with the theory of stable groups (see \cite{Poi} or \cite{WagSt} for a general reference). As before, $\mathcal{G}$ is a monster model for a theory extending a group.
	
	\subsection{Stabilizers of definable measures}
	
	\begin{definition}
		Given a measure $\mu \in \mathfrak{M}_x(\mathcal{G})$, we consider the following  (left) \emph{stabilizer group} associated to it:
 $$\Stab(\mu) := \{g \in \mathcal{G} : g \cdot \mu = \mu \} $$
 $$=\{ g \in \mathcal{G} :  \mu(\varphi(x)) = \mu(\varphi(g \cdot x)) \textrm{ for all } \varphi(x) \in \mathcal{L}(\mathcal{G}) \}.$$
	\end{definition}
%
%
	

Below we use the characterization of definability of a measure from Fact \ref{fac: chars of def meas}(3), and we follow the notation there.
	
\begin{definition} Assume that $\mu_x \in \mathfrak{M}_x(\mathcal{G})$ is definable over a small model $G \prec \mathcal{G}$.\begin{enumerate}
\item Fix a formula $\varphi(x;y) \in \mathcal{L}$ and $n \in \mathbb{N}_{>0}$. We write $\varphi'(x;y,z)$ to denote the formula $\varphi(z \cdot x;y)$, and given $i \in I_n$ we write 
$$\Phi^{\varphi', \frac{1}{n}}_{\geq i}(y,z) := \bigvee_{j \in I_n, j \geq i} \Phi^{\varphi', \frac{1}{n}}_{j}(y,z).$$
\item Consider the following formula with parameters in $G$ (where $e$ is the identity of $\mathcal{G}$):
	$$\Stab_\mu^{\varphi, \frac{1}{n}}(z) := $$
	$$\forall y \bigwedge_{i \in I_n, i \geq \frac{3}{n}} \left( \left( \Phi^{\varphi',\frac{1}{n}}_{\geq i}(y,e) \rightarrow  \Phi^{\varphi', \frac{1}{n}}_{\geq (i-\frac{2}{n})} (y,z) \right) \land \left(\Phi^{\varphi',\frac{1}{n}}_{\geq i}(y,z) \rightarrow  \Phi^{\varphi', \frac{1}{n}}_{\geq (i-\frac{2}{n})} (y,e) \right) \right).$$
	\item We define the following partial type over $G$:
	$$\Stab_\mu(z) := \bigwedge_{\varphi(x,y) \in \mathcal{L}, n \in \mathbb{N}_{>0}} \Stab^{\varphi, \frac{1}{n}}_{\mu}(z).$$

\end{enumerate}

\end{definition}

\begin{proposition}\label{prop: stab type def} Let $\mu \in \mathfrak{M}_x(\mathcal{G})$ be definable. Then
	$\Stab(\mu) = \Stab_\mu(\mathcal{G})$.
\end{proposition}
\begin{proof}

	Assume $g \notin \Stab(\mu)$. Then there exist some $\varphi(x,y) \in \mathcal{L}$ and $b \in \mathcal{G}^y$ such that  taking $r := \mu(\varphi(x,b)) =  \mu(\varphi'(x;b,e))$ and $s := \mu(\varphi(g \cdot x,b)) = \mu(\varphi'(x,b,g)) $ we have $r \neq s$. Say $r > s$ (the case $r < s$ is similar).
	We choose $n \in \mathbb{N}_{>0}$ large enough so that $|r-s| \geq  \frac{4}{n}$ (so in particular $r \geq  \frac{4}{n}$).
	As $\{ \Phi^{\varphi',\frac{1}{n}}_i(\mathcal{G}) : i \in I_n \}$ is a covering of $\mathcal{G}^{yz}$ by Fact \ref{fac: chars of def meas}(3a), there is some $i \in I_n$ such that $\models \Phi^{\varphi',\frac{1}{n}}_i(b,e)$, so in particular $\models \Phi^{\varphi',\frac{1}{n}}_{\geq i}(b,e)$.
	Hence $|r - i| < \frac{1}{n}$ by Fact \ref{fac: chars of def meas}(3b)
	(hence $i \geq \frac{3}{n}$).  If $\models \Phi^{\varphi', \frac{1}{n}}_{\geq (i-\frac{2}{n})} (b,g)$, then by Fact \ref{fac: chars of def meas}(3b) again we must have $\mu(\varphi'(x;b,g)) > i - \frac{2}{n} - \frac{1}{n}$, so $s > i - \frac{3}{n}$, and $r - s < \frac{4}{n}$, contradicting the choice of $n$. Hence $g \not \models \Stab^{\varphi, \frac{1}{n}}_\mu(z)$.
	
	Assume $g \in \Stab(\mu)$, and let $\varphi(x,y), b \in \mathcal{G}^y, n \geq 1$ and $i\geq \frac{3}{n}$ in $I_n$ be arbitrary. Assume that $\models \Phi^{\varphi',\frac{1}{n}}_{\geq i}(b,e)$ holds, then by Fact \ref{fac: chars of def meas}(3b) we have $\mu(\varphi'(x;b,e)) > i - \frac{1}{n}$. If $\models \neg \Phi^{\varphi', \frac{1}{n}}_{\geq (i-\frac{2}{n})} (b,g)$, as $\{ \Phi^{\varphi',\frac{1}{n}}_i(\mathcal{G}) : i \in I_n \}$ is a covering, we must have $\models \Phi^{\varphi',\frac{1}{n}}_j(b,g)$ for some $j < i - \frac{2}{n}$ in $I_n$. But then $\mu(\varphi'(x;b,g)) < j + \frac{1}{n}$ by Fact \ref{fac: chars of def meas}(3b) again. Hence $\mu(\varphi'(x;b,g)) < i - \frac{1}{n} < \mu(\varphi'(x;b,e))$, contradicting $g \in \Stab(\mu)$. Similarly, we get that $\models \Phi^{\varphi',\frac{1}{n}}_{\geq i}(b,g)$ implies $\models \Phi^{\varphi', \frac{1}{n}}_{\geq (i-\frac{2}{n})} (b,e)$, hence $g \models \Stab^{\varphi, \frac{1}{n}}_{\mu}(z)$ as wanted.
		\end{proof}
%
%
%
	
	\subsection{Stable groups and group chunks}\label{sec: stab grp review}
	As before, $T$ is a theory extending a group in a language $\mathcal{L}$, and we let $\mathcal{G}$ be a monster model of $T$.
In this section we review some results on stable groups that will be needed for our purpose.
	
	\begin{fact}(see e.g.~\cite[Fact 1.8]{Pillay} + \cite{CS})\label{fac: top dyn in stable}
	Let $\mathcal{G}$ be a stable group and $G \prec \mathcal{G}$ a small model. 
	Let $\mathcal{H}$ be a subgroup of $\mathcal{G}$ type-definable over $G$ (by a partial type $H(x)$ over $G$). Let $S_H(\mathcal{G}) := \{ p \in S_x(\mathcal{G}) : p(x) \vdash H(x) \}$ be the set of types over $\mathcal{G}$ concentrated on $\mathcal{H}$. Then the following hold.
	\begin{enumerate}
		\item For $p,q \in S_H(\mathcal{G})$, we have that $p \ast q$ (well-defined as all types are definable by stability) is equal to $\tp(a \cdot b / \mathcal{G})$, where $a \models p, b \models q$ in a bigger model of $T$ and $a \ind_{\mathcal{G}} b$ (in the sense of forking independence).
		\item The semigroup $(S_H(\mathcal{G}), \ast)$ has a unique minimal closed left ideal $I$ (also the unique minimal closed right ideal) which is already a subgroup of $(S_H(\mathcal{G}),*)$.
		\item $I$ is precisely the generic types of $\mathcal{H}$, and with its induced topology $I$ is a compact topological group (isomorphic to $\mathcal{H}/\mathcal{H}^{0}$).
		\item $\mathcal{H}$ admits a unique left invariant Keisler measure $\mu$  (which is also the unique right invariant Keisler measure) with $\supp(\mu) = I$. Viewing $\mu$ as a regular Borel measure on $S_H(\mathcal{G})$ and restricting to the closed set $I$, $\mu \restriction_{\supp(\mu)}$  coincides with the Haar measure on $I$.
	\end{enumerate}
	\end{fact}

 In what follows, we let $\widehat{\mathcal{G}} \succ \mathcal{G}$ be a larger monster model of $T$.
	 We will be following the notation from \cite{N3}.
	
	\begin{definition}
	\begin{enumerate}
	\item Throughout this section, $\Delta$ will denote a finite \emph{invariant set of formulas}, i.e.~formulas of the form $\varphi(u \cdot x \cdot v, \bar{y}) \in \mathcal{L}$ (so a right or a left translate of an instance of $\varphi$ is again an instance of $\varphi$).
		\item We write $R_\Delta$ to denote Shelah's $\Delta$-rank, note that it is invariant under two-sided translation since $\Delta$ is.
		\item For $P \subseteq S_x(\mathcal{G})$, we let $\cl(P)$ denote the topological closure of $P$ in $S_x(\mathcal{G})$, and $\ast P$ denote the closure of $P$ under $\ast$.
		\item For $P \subseteq S_x(\mathcal{G})$, let $\gen(P)$ denote the set of $r \in \cl(\ast P)$ such that there is no $q \in \cl(\ast P)$ with $R_{\Delta}(r) \leq R_{\Delta}(q)$ for all $\Delta$ and $R_{\Delta}(r) < R_{\Delta}(q)$ for some $\Delta$.
		\item For $P \subseteq S_x(\mathcal{G})$, let $\langle P \rangle$ denote the smallest $\mathcal{G}$-type-definable subgroup of $\widehat{\mathcal{G}}$ containing $P(\widehat{\mathcal{G}})$, where  $P(\widehat{\mathcal{G}}) = \{b \in \widehat{\mathcal{G}} : b \models p \textrm{ for some } p \in P \}$.
	\end{enumerate}
	\end{definition}
%

The following two facts are stated in \cite{N3} for strong types over $\emptyset$. Our statements here for types over $\mathcal{G}$ follow by applying them in the stable theory $T_{\mathcal{G}}$ with all of the elements of $\mathcal{G}$ named by constants (for which $\widehat{\mathcal{G}}$ is still a monster model).

	\begin{fact}\label{fac: gen in stable}
		\begin{enumerate}
		\item \cite[Fact 2.1]{N3} If $P \subseteq S_x(\mathcal{G})$ is non-empty, then $\gen(P)$ is a non-empty closed subset of $S_x(\mathcal{G})$.
		\item \cite[Lemma 2.2]{N3} $R_{\Delta}(p \ast q) \geq R_{\Delta}(p), R_{\Delta}(q)$ for any $p,q \in S_x(\mathcal{G})$ and $\Delta$ (this follows by the symmetry of forking, invariance of $R_{\Delta}$ under two-sided translations, and the fact that forking is characterized by a drop in rank).
		\end{enumerate}
	\end{fact}


	\begin{fact}\label{fac: group chunk}\cite[Theorem 2.2]{N3} Assume $T$ is stable.
		Let $P \subseteq S_x(\mathcal{G})$ be a non-empty set of types. Then 
		$$\langle P \rangle = \left\{ a \in \widehat{\mathcal{G}} : \tp(a/\mathcal{G}) \ast \gen(P) = \gen(P) \textrm{ setwise} \right\}$$
		is a $\mathcal{G}$-type-definable subgroup of $\widehat{\mathcal{G}}$ and $\gen(P)$ is precisely the set of generic types of $\langle P \rangle$ over $\mathcal{G}$.
	\end{fact}

	\subsection{Classification of idempotent measures}
	
We are ready to prove the main result of this section.

	\begin{theorem}
		Let $\mathcal{G}$ be a monster model of $T$, and let $\mu \in \mathfrak{M}_x(\mathcal{G})$ be a global Keisler measure (in particular, $\mu$ is dfs by Fact \ref{fac: props of measures relation}(3a)). Then the following are equivalent:
		\begin{enumerate}
			\item $\mu$ is idempotent;
			\item $\mu$ is the unique right-invariant (and also the unique left-invariant) measure on the type-definable subgroup $\Stab(\mu)$ of $\mathcal{G}$.
		\end{enumerate}
	\end{theorem}
	\begin{proof}
	(2) implies (1) by Proposition \ref{wow}, and we show that (1) implies (2).
	
	Let $\mu \in \mathfrak{M}_x(\mathcal{G})$ be an idempotent measure, by  Fact \ref{fac: props of measures relation}(3a) $\mu$ is definable over some small model $G \prec \mathcal{G}$ by Proposition \ref{prop: stab type def}.

By Corollary \ref{qcont}, $\supp(\mu)$ is a closed subset of $S_x(\mathcal{G})$ and is closed under $\ast$, hence $\cl(\ast \supp(\mu)) = \supp(\mu)$ and $\gen(\supp(\mu)) \subseteq \supp(\mu)$.

We claim that $\gen(\supp(\mu))$ is a two-sided ideal in $(\supp(\mu), \ast)$. Indeed, let $r \in \gen(\supp(\mu))$ and $q \in \supp(\mu)$ be arbitrary.  
If $r \ast q$ is not in $\gen(\supp(\mu))$, then there exists some $p \in \supp(\mu)$ with $R_{\Delta}(p) \geq R_{\Delta}(r \ast q) \geq R_{\Delta}(r)$ for all $\Delta$ and some inequality strict (by Fact \ref{fac: gen in stable}(2)), contradicting $r \in \gen(\supp(\mu))$.
But also if $q \ast r$ is not in $\gen(\supp(\mu))$, then there exists some $p \in \supp(\mu)$ with $R_{\Delta}(p) \geq R_{\Delta} (q \ast r) \geq R_{\Delta}(r)$ and some inequality strict, again by Fact \ref{fac: gen in stable}(2), contradicting  $r \in \gen(\supp(\mu))$. Hence $\gen(\supp(\mu)) = \supp(\mu)$ by Theorem \ref{two}.

We now fix a larger monster model $\widehat{\mathcal{G}} \succ \mathcal{G}$ as above (and view $\mathcal{G}$ as a small elementary submodel of it).
Then, by Fact \ref{fac: group chunk}, we have that 
$$\widehat{\mathcal{H}} := \langle \supp(\mu)\rangle = \{ a \in \widehat{\mathcal{G}} : a \models p \textrm{ for some } p \in \supp(\mu) \}$$
 is a $\mathcal{G}$-type-definable subgroup of $\widehat{\mathcal{G}}$ and $\supp(\mu) = \gen(\supp(\mu))$ is precisely the set of generic types of $\widehat{\mathcal{H}}$ over $\mathcal{G}$.
%
Note that the definition of $\widehat{\mathcal{H}}$ a priori uses all of the parameters in $\mathcal{G}$. We will show that it can be defined over a subset of $\mathcal{G}$ that is small with respect to $\mathcal{G}$, and that it is equal to the stabilizer of $\mu$.
Let $H(x)$ be a partial type over $\mathcal{G}$ defining $\widehat{\mathcal{H}}$, i.e.~$H(\widehat{\mathcal{G}}) = \widehat{\mathcal{H}}$. Given $p \in S_x(\mathcal{G})$, we let $\widehat{p} \in S_x(\widehat{\mathcal{G}})$ be its unique $\mathcal{G}$-definable extension, and given $\nu \in \mathfrak{M}_x(\mathcal{G})$ we let $\widehat{\nu} \in \mathfrak{M}_x(\widehat{\mathcal{G}})$ be the unique $\mathcal{G}$-definable extension of $\nu$ (by Fact \ref{fac: unique ext of def meas}).
We have the following sequence of observations.

\begin{enumerate}
	\item $p \ast q = r \iff \widehat{p} \ast \widehat{q} = \widehat{r}$ for any $p,q,r \in S_x(\mathcal{G})$.
	\item The same holds for measures, in particular $\widehat{\mu}$ is an idempotent of $\left(\mathfrak{M}_x(\widehat{\mathcal{G}}), \ast \right)$. 
	Indeed, assume $\mu, \nu \in \mathfrak{M}_x(\mathcal{G})$ are definable over some small $G' \prec \mathcal{G}$. Then $\widehat{\mu} * \widehat{\nu}$ is definable over $G'$ (by Proposition \ref{prop: pros pres under conv}) and extends $\mu * \nu$, hence $\widehat{\mu} * \widehat{\nu} = \widehat{\mu * \nu}$ by uniqueness of definable extensions (Fact \ref{fac: unique ext of def meas}).
	\item $\Stab_\mu(\widehat{\mathcal{G}}) = \Stab(\widehat{\mu})$ (by Proposition \ref{prop: stab type def} and definability of the measure).
	\item $\supp(\widehat{\mu}) = \{\widehat{p} : p \in \supp(\mu) \}$.
	
	\noindent Indeed, suppose $p \in \supp(\mu)$, but $\widehat{p} \notin \supp(\widehat{\mu})$, then there is some $\varphi(x,b) \in \widehat{p}$ such that $\widehat{\mu}(\varphi(x,b)) = 0$. In particular, $\models d_p \varphi(b)$, where $d_p(y) \in \mathcal{L}_y(c)$ for some finite tuple $c \subseteq \mathcal{G}$ is a $\varphi$-definition for $p$. By $|G|^+$-saturation of $\mathcal{G}$, we can find some $b' \in \mathcal{G}^{y}$ with $b' \equiv_{Gc} b$. By definability (and hence invariance) of $\widehat{\mu}$ over $G$, we have $\models d_p(b')$ and $ \mu(\varphi(x,b')) = \widehat{\mu}(\varphi(x,b')) = 0$. So $\varphi(x,b') \in p$, contradicting $p \in \supp(\mu)$.
	
\noindent	Conversely, suppose $r \in \supp(\widehat{\mu})$. As $\widehat{\mu}$ is definable over $G$, in particular it is non-forking over $G$, hence every type in its support is non-forking over $G$. In particular $r$ is non-forking over $G$, hence definable over $G$ by stability, so $r = \widehat{(r | _{\mathcal{G}})}$ and $r | _{\mathcal{G}}$ is clearly in $\supp(\mu)$.
	\item The generics of $H(x)$ over $\widehat{\mathcal{G}}$ are precisely $\{\widehat{p} : p \mbox{ is a generic of } H \mbox{ over } \mathcal{G}\}$. 
	
	By stability, every generic $r$ of $H(x)$ over $\widehat{\mathcal{G}}$ does not fork over $\mathcal{G}$, so it is definable over $\mathcal{G}$ and $r|_{\mathcal{G}}$ is a generic of $H(x)$ over $\mathcal{G}$, hence $r = \widehat{(r|_{\mathcal{G}})}$. Conversely, a definable (non-forking) extension of a generic type is generic.

	\item Hence, by (4) and (5), $\supp(\widehat{\mu})$ is precisely the set of the generics of $H(x)$ over $\widehat{\mathcal{G}}$, in particular $(\supp(\widehat{\mu}), \ast)$ is a topological group by Fact \ref{fac: top dyn in stable}(3).
	\item Viewed as a regular Borel measure on $S_x(\widehat{\mathcal{G}})$, $\widehat{\mu}$ restricted to $(\supp(\widehat{\mu}), \ast)$ is right-$\ast$-invariant.
	
	\noindent Indeed, $\left( \supp(\widehat{\mu}), * \right)$ is a group by (6), so for any $p \in \supp(\widehat{\mu})$, $p^{-1} \in \supp(\widehat{\mu})$ is well-defined. By regularity of the measure, it suffices to check $*$-invariance for  clopen subsets. Let $\varphi(x,\bar{b}) \in \mathcal{L}_x(\widehat{\mathcal{G}})$. Then for any $p \in \supp(\widehat{\mu})$ we have 
	$$\widehat{\mu}( \langle \varphi(x,\bar{b}) \rangle * p) = \widehat{\mu}\left(\{ q * p : \varphi(x,\bar{b}) \in q \} \right)$$
	$$ = \widehat{\mu} \left( \{q: \varphi(x, \bar{b}) \in q * p^{-1} \} \right) = \widehat{\mu}(\varphi(x \cdot c, \bar{b})),$$
	where $c \models p^{-1}|_{G \bar{b}}$. And by Corollary \ref{cor: meas on sup is invariant}, $\widehat{\mu}(\varphi(x \cdot c,\bar{b})) = \widehat{\mu}(\varphi(x,\bar{b}))$.
	
	\item By Fact \ref{fac: top dyn in stable}(4) for $\widehat{\mathcal{H}}$, there is a unique right-$\widehat{\mathcal{H}}$-invariant Keisler measure $\nu \in \mathfrak{M}_x(\widehat{\mathcal{G}})$ such that $\nu(H(x)) = 1$, $\supp(\nu)$ is the set of generics of $H(x)$ over $\widehat{\mathcal{G}}$, and $\nu \restriction_{\supp(\nu)}$ (viewed  as a Borel measure) is the Haar measure on the compact topological group $(\supp(\nu), \ast)$. 
	\item Thus $\supp(\widehat{\mu}) = \supp(\nu)$ using (6), and as both $\mu,\nu$ are right $*$-invariant restricting to $\supp(\nu)$, by uniqueness of the Haar measure we have $\widehat{\mu} \restriction_{\supp(\widehat{\mu})} = \nu \restriction_{\supp(\nu)}$, hence $\widehat{\mu} = \nu$ (as for any formula $\varphi(x) \in \cL_x(\widehat{\mathcal{G}})$ we have $\mu(\varphi(x)) = \mu(\varphi(x) \cap \supp(\mu))$ and the same for $\nu$, by Proposition \ref{prop: supp 1}(2)). 
	\item 
	By (8) we have $\widehat{\mathcal{H}} \subseteq \Stab(\nu) $ (the stabilizer in $\widehat{\mathcal{G}}$), and in fact $\widehat{\mathcal{H}} = \Stab(\nu)$ (as any two cosets of $\widehat{\mathcal{H}}$ in $\Stab(\nu)$ would give two disjoint sets of $\nu$-measure $1$). 
	Using (8) (and (3)) it follows that $\supp(\widehat\mu) \subseteq  \widehat{\mathcal{H}} = \Stab(\nu) = \Stab(\widehat{\mu}) = \Stab_{\mu}(\widehat{\mathcal{G}})$, so $\widehat \mu$ is a left- (and also right, by stability) invariant measure on the $G$-type-definable group $\Stab_\mu(\widehat{\mathcal{G}})$. Hence $\mu$ is a right-invariant measure on the  $G$-type-definable group $\Stab_\mu(\mathcal{G})$.
	\end{enumerate}
	\end{proof}
	\begin{remark}
		It was pointed out by the referee that type-definability of $\widehat{\mathcal{H}} $ over a \emph{small subset of $\mathcal{G}$} is immediate from  Hrushovski's theorem that in a stable theory $T$, any type-definable group is given by an intersection of at most $|T|$-many definable groups (see e.g.~\cite[Lemma 6.18]{pillay1996geometric}). However,  showing that $\widehat{\mathcal{H}}$ is the stabilizer appears to require some additional argument along the lines presented above.
	\end{remark}
	\begin{remark}
	Some of these results can be generalized for idempotent measures in NIP groups, and we hope to address it in future work.	
	\end{remark}

\section{Describing the convolution semigroup on finitely satisfiable measures as an Ellis semigroup}\label{sec: Ellis calc}
\subsection{Dynamics} We begin this section by recalling the construction of the Ellis semigroup. Let $X$ be a compact Hausdorff space and $S$ be a semigroup acting on $X$ by homeomorphisms. In particular, there is a map $\pi: S \times X \to X$ such that for each $s \in S$, the map $\pi_{s}:X \to X, x \mapsto \pi(s,x)$ is a homeomorphism. Let $X^X$ be the space of functions from $X$ to $X$ equipped with the product topology. Then, $\{\pi_{s}: s \in S\}$ is naturally a subset of $X^X$. Finally, the \emph{Ellis semigroup of the action $(X,S, \pi)$} is $\left(\cl \left(\left\{\pi_{s} : s \in S\right\} \right), \circ \right)$, where we take the closure of $\{ \pi_{s}: s \in S\}$ in $X^X$.  When the action map $\pi$ is clear, we will denote this  semigroup as $E(X,S)$. 

Let now $T$ be a first order theory expanding a group, $\mathcal{G}$ a saturated model of $T$, and $G \prec \mathcal{G}$ a small elementary substructure. 
Recall that $S_x(\mathcal{G},G)$ denotes the set of global types finitely satisfiable in $G$. There is a natural action of $G$ on $S_{x}(\mathcal{G},G)$ via $g \cdot p = \{\varphi(x): \varphi(g^{-1} \cdot x) \in p\}$. 

\begin{fact}[Newelski \cite{N1}]\label{fac: New isom} There exists a semigroup isomorphism (which is also a homeomorphism of compact spaces) $E(S_{x}(\mathcal{G},G),G) \cong (S_{x}(\mathcal{G},G),*)$.
\end{fact} 

In this section, we provide an analogous description for the convolution semigroup on finitely satisfiable measures in NIP theories. Recall that $\mathfrak{M}_x(\mathcal{G},G) \subseteq \mathfrak{M}_x(\mathcal{G})$ is the collection of global measures finitely satisfiable in $G$, and this space of measures is naturally identified with a closed convex subset of a real topological vector space (of all bounded real-valued measures on $S_x(\mathcal{G})$). We identify $G$ with the set $\{\delta_g : g \in G \} \subseteq \mathfrak{M}_x(\mathcal{G},G)$, and let $\conv(G)$ denote the convex hull of $G$. There is a natural semigroup action of $\conv(G)$ on $\mathfrak{M}_{x}(\mathcal{G},G)$: for any $\sum_{i=1}^{n} r_i \delta_{g_i} \in \conv(G)$ (with $g_i \in G$ and $r_i \in \mathbb{R}_{\geq 0}, \sum_{i=1}^n r_i = 1$), $\mu \in \mathfrak{M}_x(\mathcal{G},G)$ and $\varphi(x) \in \mathcal{L}_{x}(\mathcal{G})$, we define $\left( \sum_{i=1}^{n} r_i \delta_{g_i} \right) \cdot \mu \in \mathfrak{M}_x(\mathcal{G},G)$ by
\begin{equation*}
\left( \left(\sum_{i=1}^{n} r_i \delta_{g_i} \right) \cdot \mu \right) (\varphi(x)) = \sum_{i=1}^{n} r_i \mu( \varphi( g_i \cdot x)).
\end{equation*} 

For the rest of this section, we will denote elements of $\conv(G)$ as $k$, the semigroup action described above as $\pi : \conv(G) \times \mathfrak{M}_x(\mathcal{G},G) \to \mathfrak{M}_x(\mathcal{G},G)$, and the map $\mu \mapsto \pi(k,\mu)$ as  $\pi_{k}$. It is not difficult to see that for every $k \in \conv(G)$, the map $\pi_{k}$ is continuous. Therefore, we can consider the Ellis semigroup of this semigroup action, namely $E \left(\mathfrak{M}_{x}(\mathcal{G},G),\conv(G) \right)$.

We will show that if $T$ is NIP, then this Ellis semigroup is isomorphic to the convolution semigroup of global measures finitely satisfiable in $G$, i.e.~$(\mathfrak{M}_{x}(\mathcal{G},G),*)$ (Theorem \ref{thm: Ellis grp iso}). We demonstrate that that these two semigroups are isomorphic by considering the map $\rho:\mathfrak{M}_{x}(\mathcal{G},G) \to \mathfrak{M}_{x}(\mathcal{G},G)^{\mathfrak{M}_{x}(\mathcal{G},G)}$ defined by $\rho(\nu) = \rho_{\nu} := \nu * -$, and proving that the image of $\rho$ is precisely the Ellis semigroup. Before continuing, we check that $\rho$ is well-defined, and that $\mathfrak{M}_{x}(\mathcal{G},G)$ is a semigroup. 

\begin{proposition}\label{NIP:measure} Let $T$ be NIP and assume that $\mu \in \mathfrak{M}_{x}(\mathcal{G},G)$. Then:
\begin{enumerate}
\item $\mu$ is Borel-definable over $G$;
\item for any $\nu \in \mathfrak{M}_{x}(\mathcal{G},G)$, $\mu * \nu \in \mathfrak{M}_{x}(\mathcal{G},G)$;
\item the operation $*$ on $\mathfrak{M}_{x}(\mathcal{G},G)$ is associative, hence $(\mathfrak{M}_{x}(\mathcal{G},G),*)$ is a semigroup.
\end{enumerate}
\end{proposition} 

\begin{proof} $(1)$ follows from Fact \ref{fac: props of measures relation}(2a) while $(2)$ follows from Proposition \ref{prop: pros pres under conv}(2). 

 $(3)$  We show that the operation $*$ is associative on $\mathfrak{M}_{x}^{\inva}(\mathcal{\mathcal{G}},G)$ (note that it is closed under $*$, as under the NIP assumption the $\otimes$-product of two invariant measures is invariant, see e.g. \cite[Section 7.4]{Guide}). The proof is similar, but not identical, to the proof that $\otimes$ is associative on invariant measures in NIP theories \cite{GanCon2}. Fix a formula $\varphi(x) \in \mathcal{L}_{x}(\mathcal{U})$. Let $\theta(x,y;z) := \varphi(x \cdot y \cdot z)$ and $\rho(x;y,z) := \varphi(x \cdot y \cdot z)$ (where $y,z$ are variables of the same sort as $x$). Assume that $\mu, \nu, \lambda \in \mathfrak{M}_{x}^{\inva}(\mathcal{\mathcal{G}},G)$ --- all Borel-definable over $G$ by Fact \ref{fac: props of measures relation}(2a). Without loss of generality, we may assume that $G$ contains all of the parameters from $\varphi$.

Let $\widehat{\lambda}$ be a smooth extension of $\lambda|_{G}$ such that $\widehat{\lambda}$ is smooth over some small $H_1 \succeq G$. Let $\widehat{\nu}$ be a smooth extension of $\nu|_{H_1}$ such that $\widehat{\nu}$ is smooth over some small model $H_2$ such that $G \prec H_1 \prec H_2$.
We are following the notation of Section \ref{sec: def of conv}, in particular $\varphi'(x;z) = \varphi(x \cdot z)$. We have:
\begin{enumerate}[(a)]
\item $F_{(\mu * \nu)_x,G}^{\varphi'(x;z)}(r) = F_{\mu_x \otimes \nu_y,G}^{\theta(x,y;z)}(r)$ for all $r \in S_z(G)$ --- as for any $b \in \mathcal{U}^{z}$ realizing $r$ we have $F_{(\mu * \nu)_x,G}^{\varphi'}(r) = (\mu * \nu)_x(\varphi(x \cdot b)) = \mu_x \otimes \nu_y (\varphi(x \cdot y \cdot b)) = F_{\mu \otimes \nu}^{\theta}(r)$.
\item $F_{\widehat{\nu}_y\otimes\widehat{\lambda}_z,H_2}^{\rho^{*}(y,z;x)}(p)=F_{(\widehat{\nu}*\widehat{\lambda})_z, H_2}^{(\varphi')^{*}(z;x)}(p)$ for all $p \in S_x(H_2)$ --- similar to (a);
	\item $(\widehat{\nu} \otimes \widehat{\lambda})|_{G}=(\nu \otimes \lambda)|_{G}$, hence $(\widehat{\nu}*\widehat{\lambda})|_{G}=(\nu*\lambda)|_{G}$ --- by the Claim in the proof of \cite[Theorem 2.2]{GanCon2}. 
\end{enumerate} 
 Note that $\mu$ is invariant over both $G$ and $H_2$, hence for any 
measure $\gamma \in \mathfrak{M}_{yz}(\mathcal{\mathcal{G}})$ and formula $\psi(x;y,z) \in \mathcal{L}(G)$
 we have (as in Proposition \ref{prop: int over diff mod}):
\begin{enumerate}
  \item[(d)]  $ \int_{S_{yz}(G)}F_{\mu,G}^{\psi}d(\gamma|_{G}) = \int_{S_{yz}(H_2)}F_{\mu,H_2}^{\psi}d(\gamma|_{H_2})$.
\end{enumerate}
Finally, we also have:
\begin{enumerate}
	\item[(e)] 	$\widehat{\nu} \otimes \widehat{\lambda}$ is smooth over $H_2$ (by Fact \ref{fac: meas commute}(4));
	\item[(f)] $\widehat{\nu} * \widehat{\lambda}$ is dfs over $H_2$ (by Proposition \ref{prop: pros pres under conv}(1) and (2));
\end{enumerate}
Using these observations we calculate:
\begin{gather*}
	[(\mu*\nu)*\lambda](\varphi(x))=[(\mu*\nu)_x\otimes\lambda_z](\varphi(x\cdot z))=\int_{S_{z}(G)}F_{\mu*\nu,G}^{\varphi'(x;z)}d(\lambda_z|_{G})\\
	\overset{\textrm{(a)}}{=}\int_{S_{z}(G)}F_{\mu_x\otimes\nu_y,G}^{\theta(x,y;z)}d(\lambda_z|_{G})\\=[(\mu_x\otimes\nu_y)\otimes\lambda_z](\theta(x,y;z))=[(\mu_x \otimes \nu_y)\otimes\lambda_z](\varphi(x\cdot y\cdot z))\\
	=[\mu_x\otimes(\nu_y\otimes\lambda_z)](\varphi(x\cdot y\cdot z))\textrm{ (by Fact \ref{fac: meas commute}(5))}\\
	=\int_{S_{yz}(G)}F_{\mu_x,G}^{\rho(x;y,z)}d \left((\nu_y\otimes\lambda_z)|_{G} \right) \overset{\textrm{(c)}}{=} \int_{S_{yz}(G)}F_{\mu_x,G}^{\rho(x;y,z)}d \left((\widehat{\nu}_y\otimes \widehat{\lambda}_z)|_{G} \right)\\
	=\int_{S_{yz}(H_{2})}F_{\mu_x,H_2}^{\rho(x;y,z)}d \left((\widehat{\nu}_y\otimes\widehat{\lambda}_z)|_{H_2} \right) \\
	\textrm{(by (d) with }\gamma_{y,z} := \widehat{\nu}_y \otimes \widehat{\lambda}_z \textrm{ and } \psi(x;y,z) :=\rho(x;y,z) \textrm{)} \\
	=\int_{S_{x}(H_{2})}F_{\widehat{\nu}_y\otimes\widehat{\lambda}_z, H_2}^{\rho^{*}(y,z;x)}d(\mu_x|_{H_2})  \textrm{ (by (e) and Fact \ref{fac: meas commute}(1))}\\
	\overset{\textrm{(b)}}{=}\int_{S_{x}(H_{2})}F_{(\widehat{\nu}*\widehat{\lambda})_z,H_2}^{(\varphi')^{*}(z;x)}d(\mu_x|_{H_2})\\
=\int_{S_{z}(H_{2})}F_{\mu_x,H_2}^{\varphi'(x;z)}d \left((\widehat{\nu}*\widehat{\lambda})_z|_{H_2} \right) \textrm{(by (f) and Fact \ref{fac: meas commute}(3))}\\
=\int_{S_{z}(G)}F_{\mu_x,G}^{\varphi'(x;z)}d \left((\widehat{\nu}*\widehat{\lambda})_z|_{G} \right) \textrm{ (by (d) with } \gamma_z:=(\widehat{\nu} * \widehat{\lambda} )_z\textrm{ and } \psi(x;z):=\varphi'(x;z) \textrm{)} \\
	\overset{\textrm{(c)}}{=}\int_{S_{z}(G)}F_{\mu_x, G}^{\varphi'(x;z)}d \left((\nu*\lambda)_z|_{G} \right) \\
	=\mu*(\nu*\lambda)(\varphi(x)).
\end{gather*}
\end{proof}

Hence the map  $\rho: \mathfrak{M}_{x}(\mathcal{G},G) \to \mathfrak{M}_{x}(\mathcal{G},G)^{\mathfrak{M}_{x}(\mathcal{G},G)}$ is well-defined. In the next subsection we show that it is also left-continuous. 

\subsection{Left-continuity of convolution}\label{sec: left-cont of conv}
We begin with a general continuity result in arbitrary NIP theories. Let $T$ be an NIP theory, $\mathcal{U}$ a monster model of $T$, and $M$ a small elementary substructure of $\mathcal{U}$.

\begin{proposition}[T NIP]\label{continuity} Let $M \prec \mathcal{U}$ and let  $\mathfrak{M}_{x}^{\inva}(\mathcal{U},M)$ be the closed set of global $M$-invariant measures (Definition \ref{def: props of measures}). If $\nu \in \mathfrak{M}_{y}(\mathcal{U})$ and $\varphi(x;y)$ is any partitioned $\mathcal{L}_{xy}(\mathcal{U})$ formula, then the map $-\otimes\nu(\varphi(x;y)):\mathfrak{M}^{\inva}_{x}(\mathcal{U},M)\to[0,1]$ is continuous. 
\end{proposition}

\begin{proof}
Choose $N_{0} \prec \cU$ small and such that $M \preceq N_{0}$, and $N_{0}$ contains the parameters of $\varphi$. Then, choose a small $N \prec \cU$ such that $N_{0} \preceq N$ and there exists some $\widehat{\nu}\in\mathfrak{M}_{y}(\mathcal{U})$ such that $\widehat{\nu}|_{N_{0}}=\nu|_{N_{0}}$ and $\widehat{\nu}$ is smooth over $N$ (by Fact \ref{fac: NIP ext to smooth}). Fix $\varepsilon \in \mathbb{R}_{>0}$, by Fact \ref{fac: props of measures relation}(1a) let $\overline{b} = (b_1, \ldots, b_n)$ be some $(\varphi^{*},\varepsilon)$-approximation for $\widehat{\nu}$ over $N$ (where $\varphi^*(y;x) = \varphi(x;y)$ and $\overline{b}$ is some element in $(N^{y})^{<\omega}$, see Definition \ref{def: props of measures}(7)). Note that every $\mu\in\mathfrak{M}^{\inva}_{x}(\mathcal{U},M)$ is invariant over both $N_{0}$ and $N$. Then we have (the last equality holds as in the proof of Proposition \ref{prop: int over diff mod}):

\begin{equation*}
    \mu\otimes\nu(\varphi(x;y))=\int_{S_{y}(N_{0})}F_{\mu,N_{0}}^{\varphi}d(\nu|_{N_{0}})=\int_{S_{y}(N_{0})}F_{\mu,N_{0}}^{\varphi}d(\widehat{\nu}|_{N_{0}})=\int_{S_{y}(N)}F_{\mu,N}^{\varphi}d(\widehat{\nu}|_{N}).
\end{equation*}
As $\widehat{\nu}$ is smooth over $N$, by Fact \ref{fac: meas commute}(1) we have
\begin{equation*} \int_{S_{y}(N)}F_{\mu,N}^{\varphi}d(\widehat{\nu}|_{N}) =  \int_{S_{x}(N)}F_{\widehat{\nu},N}^{\varphi^{*}}d(\mu|_{N}).
\end{equation*} 

Note that $F^{\varphi^*}_{\Av_
{\overline{b}}, N} (p) = \frac{1}{n}\sum_{i=1}^{n} \chi_{\{ r \in S_x(N) : \varphi(x,b_i) \in r \}}(p)$ for every $p \in S_x(N)$, where $\chi$ is the characteristic function.  Now, using that $\bar{b} \subseteq N$ is a $(\varphi^{*},\varepsilon)$-approximation  for $\widehat{\nu}$, we have the following (note that we identify $\varphi(x,b_i)$ with the set of types satisfying it over $N$ in the first step, and over $\cU$ in the second step).
$$
    \int_{S_{x}(N)}F_{\widehat{\nu},N}^{\varphi^{*}}d(\mu|_{N}) \approx_{\varepsilon}\int_{S_{x}(N)}F_{\Av_{\overline{b}}, N}^{\varphi^{*}}d(\mu|_{N})$$
    $$ = \int_{S_x(N)} \left(  \frac{1}{n} \sum_{i=1}^{n} \chi_{\varphi(x,b_i)} \right) d(\mu|_N) $$
    $$ = \frac{1}{n}\sum_{i=1}^{n} \left( \int_{S_x(N)} \chi_{\varphi(x,b_i)}  d(\mu|_N)  \right) = \frac{1}{n}\sum_{i=1}^{n} \mu|_N(\varphi(x,b_i))$$
    $$ =\frac{1}{n}\sum_{i=1}^{n}\int_{S_{x}(\mathcal{U})}\chi_{\varphi(x,b_i)}  d\mu.
$$
 Clearly, each map $\int \chi_{\varphi(x,b_{i})}:\mathfrak{M}_{x}(\mathcal{U})\to[0,1]$ is continuous by the definition of the topology on the space of measures. Therefore, each map $\int\chi_{\varphi(x,b_{i})}:\mathfrak{M}^{\inva}_{x}(\mathcal{U},M)\to[0,1]$ is continuous, hence their sum is continuous as well. Since the choice of $\overline{b}$ is independent of the choice of $\mu$, we have
\begin{equation*}
    \sup_{\mu\in\mathfrak{M}^{\inva}_{x}(\mathcal{U},M)} \left \lvert \mu\otimes\nu(\varphi(x;y))-\frac{1}{n}\sum_{i=1}^{n}\int_{S_{x}(\mathcal{U})}\chi_{\varphi(x,b_{i})}d\mu \right \rvert<\varepsilon.
\end{equation*}
Therefore, the map $-\otimes\nu(\varphi(x;y))$ is a uniform limit of continuous functions and hence continuous.
\end{proof}

Now, we apply this to our group theoretic context. Let again $T$ be an NIP theory expanding a group, $\mathcal{G}$ a monster model of $T$, and $G \prec \mathcal{G}$ a small model. 

\begin{proposition}\label{prop: conv is left cont} Let $\nu \in \mathfrak{M}_{x}(\mathcal{G},G)$. Then the map $-*\nu:\mathfrak{M}_{x}(\mathcal{G},G) \to \mathfrak{M}_{x}(\mathcal{G},G)$ is continuous. 
\end{proposition} 
\begin{proof} Let $U$ be a basic open subset of $\mathfrak{M}_{x}(\mathcal{G},G)$. That is, there exist formulas $\varphi_{1}(x),...,\varphi_{n}(x)$ in $\mathcal{L}_{x}(\mathcal{G})$ and real numbers $r_1,...,r_n,s_1,...,s_n \in [0,1]$ such that
\begin{equation*}
U = \bigcap_{i=1}^{n} \{\mu \in \mathfrak{M}_{x}(\mathcal{G},G): r_i < \mu(\varphi_{i}(x)) < s_i\}.
\end{equation*} 
Then we have
\begin{equation*} \Big(-*\nu\Big)^{-1}(U) = \bigcap_{i=1}^{n}\{\mu \in \mathfrak{M}_{x}(\mathcal{G},G): r_i < \mu * \nu(\varphi_{i}(x)) < s_i\}
\end{equation*} 
\begin{equation*} = \bigcap_{i=1}^{n}\{\mu \in \mathfrak{M}_{x}(\mathcal{G},G): r_i < \mu_{x} \otimes \nu_{y}(\varphi_{i}(x \cdot y)) < s_{i}\}
\end{equation*} 
\begin{equation*} = \bigcap_{i=1}^{n}\big(- \otimes \nu_{y}  \left(\varphi_{i}(x\cdot y)\right) \big)^{-1} \Big(\{\mu \in \mathfrak{M}_{x}(\mathcal{G},G): r_i < \mu(\varphi_{i}(x)) < s_i  \} \Big).
\end{equation*} 
Therefore, by continuity of the map $-\otimes\nu(\varphi(x \cdot y))$ (Proposition \ref{continuity}), the preimage of $U$ under $-* \nu$ is a finite intersection of open sets, and therefore open. 
\end{proof}

\subsection{The isomorphism}
In this subsection we show that the map $\rho:\mathfrak{M}_{x}(\mathcal{G},G) \to E(\mathfrak{M}_{x}(\mathcal{G},G), \conv(G))$ given by $\rho(\nu) = \rho_{\nu} = \nu * -$ is an isomorphism. We begin by recalling the topology on $\mathfrak{M}_{x}(\mathcal{G},G)^{\mathfrak{M}_{x}(\mathcal{G},G)}$.

\begin{remark}\label{rem: openset} The topology on $\mathfrak{M}_{x}(\mathcal{G},G)^{\mathfrak{M}_{x}(\mathcal{G},G)}$ is generated by the basic open sets of the form 
\begin{equation*}
    U = \bigcap_{i=1}^n \{f: \mathfrak{M}_{x}(\mathcal{G},G) \to \mathfrak{M}_{x}(\mathcal{G},G) \mid r_i < f(\nu_i)(\psi_i(x)) < s_i\},
\end{equation*}
with $n \in \mathbb{N}$, $r_i, s_i \in \mathbb{R}$, $\psi_i(x) \in \mathcal{L}_{x}(\mathcal{G})$, and $\nu_i \in \mathfrak{M}_{x}(\mathcal{G},G)$ (with possible repetitions of $\nu_i$'s and $\psi_i$'s).
\end{remark}

\begin{lemma}\label{lem: rho inj} The map $\rho$ is injective.
\end{lemma}

\begin{proof} Note that for every $\nu \in \mathfrak{M}_{x}(\mathcal{G},G)$, $\rho_{\nu}(\delta_{e}) = \nu$, where $e$ is the identity of $\mathcal{G}$. 
\end{proof}

\begin{lemma}\label{lem: image of rho} If $\mu \in \mathfrak{M}_{x}(\mathcal{G},G)$, then $\rho_{\mu} \in \cl \big(\{\pi_{k}: k \in \conv(G)\} \big)$. So 
$$\rho\left(\mathfrak{M}_{x}(\mathcal{G},G) \right) \subseteq E(\mathfrak{M}_{x}(\mathcal{G},G), \conv(G)).$$
\end{lemma}

\begin{proof} Let $U$ be an open subset of $\mathfrak{M}_{x}(\mathcal{G},G)^{\mathfrak{M}_{x}(\mathcal{G},G)}$ containing $\rho_{\mu}$. It is a union of basic open sets (see Remark \ref{rem: openset}), hence we can choose some $n \in \mathbb{N}$, a sufficiently small $\varepsilon > 0$ and some $\psi_{1}(x),...,\psi_{n}(x) \in \mathcal{L}_{x}(\mathcal{U})$ and $\nu_1,...,\nu_n \in \mathfrak{M}_{x}(\mathcal{G},G)$ such that
\begin{equation*}
   B_{\varepsilon} :=  \bigcap_{i=1}^n \{f  : |f(\nu_i)(\psi_i(x)) - \rho_{\mu}(\nu_{i})(\psi_i(x))| < \varepsilon\}  \subseteq U.
\end{equation*}
Let $H_0 \prec \mathcal{G}$ be a small model containing $G$ and the parameters of $\psi_1, \ldots, \psi_n$.
By Fact \ref{fac: unique ext of def meas}, we can choose a small model $H \prec \mathcal{G}$ and measures $\widehat{\nu}_i \in \mathfrak{M}_x(\mathcal{G})$ such that:
\begin{itemize}
	\item  $G \preceq H_0 \preceq H \prec\mathcal{G}$;
	\item $\widehat{\nu}_i|_{H_0} = \nu_i|_{H_0}$, for all $1 \leq i \leq n$;
	\item $\widehat{\nu}_i$ is smooth over $H$, for all $1 \leq i \leq n$.
\end{itemize}
Take some $0 < \varepsilon_{0} < \frac{\varepsilon}{3}$. Recall from Section \ref{sec: def of conv}  that $\psi'(x;y) = \psi(x \cdot y) \in \mathcal{L}_{xy}(H_0)$.
By Fact  \ref{fac: props of measures relation}(1a), let $\overline{b}_{i} = (b_{i,j} : 1 \leq j \leq m_i) \in H^{<\omega}$ be a $(\left(\psi'_i \right)^*,\varepsilon_{0})$-approximation for $\widehat{\nu}_{i}$. Then, using that $\mu$ is invariant over both $H_0$ and $H$ and $\widehat{\nu}_i$ is smooth over $H$ as in Proposition \ref{continuity}, for every $1 \leq i \leq n$ we have: 
\begin{equation*}
    \rho_{\mu}(\nu_{i})(\psi_{i}(x))=\mu*\nu_{i}(\psi_{i}(x))=\mu\otimes\nu_{i}(\psi_{i}(x\cdot y))
\end{equation*}
\begin{equation*}
    =\int_{S_{y}(H_0)}F_{\mu,H_0}^{\psi'_{i}}d(\nu_{i}|_{H_0})=\int_{S_{y}(H_0)}F_{\mu,H_0}^{\psi'_{i}}d(\widehat{\nu}_{i}|_{H_0})
\end{equation*}
\begin{equation*}
    =\int_{S_{y}(H)}F_{\mu,H}^{\psi'_{i}}d(\widehat{\nu}_{i}|_{H})=\int_{S_{x}(H)}F_{\widehat{\nu}_{i},H}^{\left(\psi'_{i} \right)^{*}}d(\mu|_{H})
\end{equation*}
\begin{equation*}
    \approx_{\varepsilon_{0}}\int_{S_{x}(H)}F_{\Av_{\overline{b}_{i}}, H}^{\left(\psi'_{i}\right)^{*}}d(\mu|_{H})=\frac{1}{m_i}\sum_{j=1}^{m_i}\mu(\psi_{i}(x\cdot b_{i,j})).
\end{equation*}
Let $\Psi= \{\psi_{i}(x \cdot b_{i,j}) : 1 \leq i \leq n, 1 \leq j \leq m_i\}$.  Since $\mu$ is finitely satisfiable in $G$, we can find some $k_{\mu} \in  \conv(G)$ such that $k_{\mu}(\theta(x)) \approx_{\varepsilon_0} \mu(\theta(x))$ for each $\theta(x) \in \Psi$ (see the proof of Proposition \ref{convfs}). We claim that then $\pi_{k_{\mu}}$ is in $B_{\varepsilon}$. This follows directly from running the equations above in reverse: for each $1 \leq i \leq n$ we have (using that $k_\mu$ is obviously invariant over $G$, hence also over $H_0$)
\begin{equation*}
    \frac{1}{m_i}\sum_{j=1}^{m_i}\mu(\psi_{i}(x\cdot b_{i,j})) \approx_{\varepsilon_0} \frac{1}{m_i}\sum_{j=1}^{m_i}k_{\mu}(\psi_{i}(x\cdot b_{i,j})) 
\end{equation*}
\begin{equation*}
    = \int_{S_{x}(H)}F_{\Av_{\overline{b}_i}, H}^{\left(\psi'_{i} \right)^{*}}d(k_{\mu}|_{H}) \approx_{\varepsilon_{0}} \int_{S_{x}(H)}F_{\widehat{\nu}_{i}, H}^{\left(\psi'_{i} \right)^*}d \left(k_{\mu}|_{H} \right)
\end{equation*}
\begin{equation*}
 =\int_{S_{y}(H)}F_{k_{\mu}, H}^{\psi'_{i}}d(\widehat{\nu}_{i}|_{H})=\int_{S_{y}(H_0)}F_{k_{\mu}, H_0}^{\psi'_{i}}d(\widehat{\nu}_{i}|_{H_0})
\end{equation*}
\begin{equation*}
    =k_{\mu} \otimes \nu_i(\psi_i(x \cdot y)) = \pi_{k_{\mu}}(\nu_{i})(\psi_i(x)).
\end{equation*}
Hence $\rho_{\mu}(\nu_{i})(\psi_{i}(x)) \approx_{3\varepsilon_{0}}\pi_{k_{\mu}}(\nu_{i})(\psi_i(x))$ for each $1 \leq i \leq n$, so $\pi_{k_\mu} \in B_{\varepsilon} \subseteq U$ and we are finished.
\end{proof}

\begin{lemma}\label{lem: surj} $\rho\left(\mathfrak{M}_{x}(\mathcal{G},G) \right) = E(\mathfrak{M}_{x}(\mathcal{G},G), \conv(G))$.
\end{lemma}

\begin{proof} Let $f \in E\left(\mathfrak{M}_{x}(\mathcal{G},G),\conv(G) \right)$ be arbitrary. Then $f \in \cl\left(\{\pi_{k}: k \in \conv(G)\} \right)$, and so there exists a net $(k_{i})_{i \in I}$ with $k_i \in \conv(G)$ such that $\lim_{i \in I} \pi_{k_i} = f$. Then, using Remark \ref{rem: openset}, for every $\psi(x) \in \mathcal{L}_{x}(\mathcal{G})$ and $\nu \in \mathfrak{M}_{x}(\mathcal{G},G)$ we  have
\begin{equation*}
    \lim_{i \in I} \pi_{k_i}(\nu)(\psi(x)) = f(\nu)(\psi(x)).
\end{equation*}
Consider $\delta_{e}$, where $e \in \mathcal{G}$ is the identity. Let $\mu_{f} := f(\delta_{e}) \in \mathfrak{M}_{x}(\mathcal{G},G)$. We claim that the net $(k_i)_{i \in I}$ converges to $\mu_{f}$ in $\mathfrak{M}_{x}(\mathcal{G},G)$. Indeed, for any $\psi(x) \in \mathcal{L}_{x}(\mathcal{G})$ we have
\begin{equation*}
    \lim_{i \in I} {k_i}(\psi(x))=\lim_{i \in I} \pi_{k_i}(\delta_{e})(\psi(x)) = f(\delta_{e})(\psi(x)) = \mu_f(\psi(x)).
\end{equation*}
Next, we claim that for any $\nu \in \mathfrak{M}_{x}(\mathcal{G},G)$, we have that $f(\nu) = \rho_{\mu_f}(\nu)$. Indeed, first we have
\begin{equation*}
    f(\nu) = \lim_{i \in I}\pi_{k_i}(\nu) = \lim_{i \in I}[\pi_{k_i} \circ \rho_{\nu}](\delta_{e}) = \lim_{i \in I} \rho_{k_i * \nu}(\delta_{e}) = \lim_{i \in I} [k_i * \nu].
\end{equation*}
The map $-*\nu: \mathfrak{M}_{x}(\mathcal{G},G) \to \mathfrak{M}_{x}(\mathcal{G},G)$ is continuous by Proposition \ref{prop: conv is left cont}, hence it  commutes with net limits. Therefore,
\begin{equation*}
\lim_{i \in I} [k_i * \nu] = [\lim_{i \in I} k_i] * \nu  = \mu_{f} * \nu = \rho_{\mu_{f}}(\nu).
\end{equation*}
We conclude that $f = \rho_{\mu_f} = \mu_f*-$.
\end{proof} 

\begin{lemma}\label{lem: rho inv cont} The map $\rho^{-1}: E(\mathfrak{M}_{x}(\mathcal{G},G),\conv(G)) \to \mathfrak{M}_{x}(\mathcal{G},G)$ is a continuous bijection.
\end{lemma} 
\begin{proof} The map $\rho^{-1}$ is a well-defined bijection by Lemmas \ref{lem: rho inj} and \ref{lem: surj}.
Let $U$ be a basic open subset of $\mathfrak{M}_{x}(\mathcal{G},G)$, say 
\begin{equation*}U = \bigcap_{i=1}^{n} \{\mu \in \mathfrak{M}_{x}(\mathcal{G},G): r_i < \mu(\varphi_{i}(x)) <s_i\}
\end{equation*}
for some $n \in \mathbb{N}$, $\varphi_{i}(x) \in \mathcal{L}_{x}(\mathcal{U})$ and $r_i,s_i \in [0,1]$.
Then, 
\begin{equation*} \left(\rho^{-1} \right)^{-1}(U) = \bigcap_{i=1}^{n}\{ f \in E(\mathfrak{M}_{x}(\mathcal{G},G),\conv(G)) : r_i<f(\delta_{e})(\varphi_{i}(x)) < s_i\}.
\end{equation*}
This is a restriction of a basic open subset (see Remark \ref{rem: openset}) to $E(\mathfrak{M}_{x}(\mathcal{G},G),\conv(G))$, hence open in the subspace topology. 
\end{proof}

\begin{theorem}\label{thm: Ellis grp iso} The map $\rho: (\mathfrak{M}_{x}(\mathcal{G},G),*) \to E(\mathfrak{M}_{x}(\mathcal{G},G),\conv(G))$ is a homeomorphism which respects the semigroup operation, and therefore an isomorphism.
\end{theorem} 

\begin{proof} The map $\rho$ is a homeomorphism since, by Lemma \ref{lem: rho inv cont}, $\rho^{-1}$ is a continuous bijection between compact Hausdorff spaces.  And note that $\rho(\mu * \nu)(\lambda) = (\mu * \nu) * \lambda = \mu * (\nu * \lambda) = \rho_{\mu} ( \nu * \lambda) = \rho_{\mu} \circ \rho_{\nu}(\lambda)$, hence $\rho(\mu * \nu) = \rho_{\mu} \circ \rho_{\nu}$.
\end{proof} 

\begin{remark}\label{rem : without conv}
	On the other hand, if $T$ is NIP, then 
		$$E(\mathfrak{M}_{x}(\mathcal{G},G),G) \cong E(S_{x}(\mathcal{G},G),G),$$
	
	\noindent and so $\cong (S_{x}(\mathcal{G},G),*)$ by Fact \ref{fac: New isom}. For a countable $G \prec \mathcal{G}$, this is an immediate consequence of the corresponding observation in the context of tame metrizable dynamical systems (see e.g.~\cite[Theorem 1.5]{glasner2006tame}); and for an arbitrary small $G \prec \mathcal{G}$, an approximation argument with smooth measures (as in Lemma \ref{lem: image of rho}) can be adapted. As typically $(\mathfrak{M}_{x}(\mathcal{G},G),*) \not \cong (S_{x}(\mathcal{G},G),*)$, we see that it was crucial to consider the action of $\conv(G)$ rather than $G$ in our characterization of $(\mathfrak{M}_{x}(\mathcal{G},G),*)$ as an Ellis semigroup.
	
\end{remark}

\bibliographystyle{plain}
\bibliography{refs}

\end{document}